\documentclass[11pt,reqno]{amsart}

\usepackage{hyphenat}
\usepackage[a4paper,margin=31mm]{geometry}
\usepackage[foot]{amsaddr}
\usepackage{amsmath,amssymb,amsthm,mathtools}
\usepackage{microtype}
\usepackage{enumitem}

\usepackage{xcolor}
\usepackage[hypertexnames=false]{hyperref}
\usepackage{bookmark}
\usepackage{aliascnt}
\usepackage{silence}
\mathtoolsset{showonlyrefs}
\numberwithin{equation}{section}
\allowdisplaybreaks
\usepackage{mathrsfs}

\usepackage[section]{placeins}
\usepackage{etoolbox}
\AtBeginEnvironment{thebibliography}{\raggedright}

\hypersetup{
  colorlinks=true,
  linkcolor=blue!55!black,
  citecolor=green!45!black,
  urlcolor=blue!65!black,
  pdftitle={Positivity Rigidity for Grossman--Larson Characters and the Kingman Face of the Hoffman Rooted-Tree Graph},
  pdfauthor={Shengjun Zhang},
  pdfsubject={Positive Grossman--Larson characters, two-sided Pieri states, and the Kingman face},
  pdfkeywords={Grossman--Larson Hopf algebra, basis-positive characters,
two-sided Pieri states, symmetric functions, Kingman paintboxes,
central measures, Martin boundary}
}

\newtheorem{theorem}{Theorem}[section]
\newaliascnt{proposition}{theorem}
\newtheorem{proposition}[proposition]{Proposition}
\aliascntresetthe{proposition}
\newaliascnt{lemma}{theorem}
\newtheorem{lemma}[lemma]{Lemma}
\aliascntresetthe{lemma}
\newaliascnt{corollary}{theorem}
\newtheorem{corollary}[corollary]{Corollary}
\aliascntresetthe{corollary}
\newaliascnt{definition}{theorem}
\newtheorem{definition}[definition]{Definition}
\aliascntresetthe{definition}
\newaliascnt{remark}{theorem}
\newtheorem{remark}[remark]{Remark}
\aliascntresetthe{remark}
\theoremstyle{plain}
\newaliascnt{example}{theorem}
\newtheorem{example}[example]{Example}
\aliascntresetthe{example}

\usepackage{tikz}

\tikzset{
  rtreepic/.style={baseline=-0.6ex,scale=.42,every node/.style={circle,fill,inner sep=1.15pt},every path/.style={line width=.42pt}}
}
\newcommand{\TreeTauPic}{\begin{tikzpicture}[rtreepic]
  \node (a) at (0,0) {}; \node (b) at (0,-1) {}; \draw (a)--(b);
\end{tikzpicture}}
\newcommand{\TreeCthreePic}{\begin{tikzpicture}[rtreepic]
  \node (a) at (0,0) {}; \node (b) at (0,-1) {}; \node (c) at (0,-2) {}; \draw (a)--(b)--(c);
\end{tikzpicture}}
\newcommand{\TreeCfourPic}{\begin{tikzpicture}[rtreepic]
  \node (a) at (0,0) {}; \node (b) at (0,-.8) {}; \node (c) at (0,-1.6) {}; \node (d) at (0,-2.4) {}; \draw (a)--(b)--(c)--(d);
\end{tikzpicture}}
\newcommand{\TreeYPic}{\begin{tikzpicture}[rtreepic]
  \node (a) at (0,0) {}; \node (b) at (0,-.9) {}; \node (c) at (-.5,-1.8) {}; \node (d) at (.5,-1.8) {}; \draw (a)--(b)--(c); \draw (b)--(d);
\end{tikzpicture}}
\newcommand{\TreeBPic}{\begin{tikzpicture}[rtreepic]
  \node (a) at (0,0) {}; \node (b) at (-.45,-.9) {}; \node (c) at (.45,-.9) {}; \node (d) at (.45,-1.8) {}; \draw (a)--(b); \draw (a)--(c)--(d);
\end{tikzpicture}}
\newcommand{\T}{\mathcal T}
\newcommand{\GL}{\mathcal G}
\newcommand{\Ipf}{\mathcal I_{\mathrm{pf}}}
\newcommand{\GLpf}{\mathcal G_{\mathrm{pf}}}
\newcommand{\Aut}{\operatorname{Aut}}
\newcommand{\Leaves}{\operatorname{Leaves}}
\newcommand{\Prob}{\mathbb P}
\newcommand{\E}{\mathbb E}
\newcommand{\N}{\mathbb N}
\newcommand{\R}{\mathbb R}
\newcommand{\one}{{\mathord{\bullet}}}
\newcommand{\Span}{\operatorname{span}}
\newcommand{\QSL}{\operatorname{QSL}}
\newcommand{\SL}{\operatorname{SL}}
\newcommand{\King}{\nabla_\infty}
\newcommand{\Tcen}{\mathcal T_{\mathrm{cen}}}

\title[Grossman--Larson characters and the Kingman face]
{Positivity Rigidity for Grossman--Larson Characters and the Kingman Face of the Hoffman Rooted-Tree Graph}
\address{Université Paris-Saclay, Faculté des Sciences d’Orsay, Institut de mathématiques d’Orsay, Bâtiment 307, F-91405 Orsay, France}
\email{zhang.shengjun@universite-paris-saclay.fr}
\author{Shengjun Zhang}

\begin{document}
\raggedbottom

\begin{abstract}
We classify the real characters of the Grossman--Larson Hopf algebra that are
nonnegative on the rooted-tree basis. They vanish on trees with branching
away from the root, and the normalized characters are parametrized by Kingman
paintboxes. Two-sided Pieri states are mixtures of the normalized
characters, with unique mixing measures. The corresponding laws 
form a proper exposed Bauer face of the simplex of
central measures on the Hoffman rooted-tree graph. For the path-forest subgraph, 
the full and minimal Martin boundaries coincide
and are homeomorphic to the Kingman simplex. The limiting ranked root-branch
frequencies generate the completed central tail.
\end{abstract}

\maketitle
\setcounter{tocdepth}{1}
\tableofcontents

\section{Introduction}\label{sec:introduction}

The Grossman--Larson product combines rooted trees by grafting all root
branches of one tree onto vertices of another. Its structure constants are
nonnegative integers, and left multiplication by the two-vertex chain is the
leaf-growth operator of the Hoffman graph. Consequently, a character that is
nonnegative on the tree basis and takes the value one on this chain defines a
central path measure. We classify the measures obtained in this way and
characterize their mixtures through a two-sided Pieri identity.

Let \(\GL\) be the real Grossman--Larson Hopf algebra, with basis
\(F_t\) indexed by finite unlabelled non-plane rooted trees. Write
\(C_m\) for the chain with \(m\) vertices, with
\(C_1=\one\) and \(C_2=\tau\). Let \(n(t;T)\) be the number
of vertices of \(t\) at which attaching a leaf produces \(T\), and
write \(t\nearrow T\) when \(n(t;T)>0\). We use the convention that
the root of the left factor is removed and its branches are grafted onto the
right factor. Then
\begin{equation}\label{eq:intro-Pieri}
 F_\tau F_t=\sum_{T:t\nearrow T}n(t;T)F_T.
\end{equation}

By a character we mean a multiplicative real linear functional,
with the zero map allowed. 
A linear functional is \emph{basis-positive} if it is nonnegative on every
\(F_t\). A normalized positive left Pieri state is a basis-positive linear
functional \(L\) satisfying \(L(F_\one)=1\) and
\(L(F_\tau x)=L(x)\) for all \(x\in\GL\). Such states correspond to
central measures on the Hoffman graph: finite weighted paths with the same
endpoint have equal cylinder probabilities. We denote the central law of
\(L\) by \(\mu_L\). The normalized basis-positive characters are the
multiplicative left Pieri states.

The proof of rigidity begins with the smallest tree having a branching
vertex away from the root. Let \(Y\) have a root with one child, which in
turn has two leaf children. Direct grafting gives
\begin{equation}\label{eq:intro-commutator}
 F_\tau F_{C_3}-F_{C_3}F_\tau=F_Y.
\end{equation}
Every real character sends this commutator to zero. Basis positivity then
propagates the conclusion: if a product of tree basis vectors has value zero,
every basis vector appearing with nonzero coefficient also has value zero.
We show that, starting from \(Y\), every tree with branching away from the
root is reached in at most two such product steps.

A \emph{path-forest tree} is a rooted tree whose non-root vertices have at
most one child. Write \(P_\lambda\) for the tree with chain branches of
sizes given by an integer partition \(\lambda\): a part \(\lambda_i\)
corresponds to a branch with \(\lambda_i\) vertices.
Figure~\ref{fig:path-forest-trees} illustrates the definition.

\begin{figure}[htbp]
\centering
\begin{tikzpicture}[
  x=1cm,y=1cm,
  line width=.45pt,
  font=\small,
  pf vertex/.style={
    circle,fill=black,inner sep=0pt,minimum size=4pt
  },
  pf root/.style={pf vertex,minimum size=5pt}
]
  % A path-forest tree with branch sizes 3, 2, 2, 1.
  \begin{scope}
    \node[pf root,label=above:{root}] (r) at (0,0) {};

    \foreach \name/\x/\y in {
      a1/-1.5/-.6, a2/-1.5/-1.2, a3/-1.5/-1.8,
      b1/-.5/-.6,  b2/-.5/-1.2,
      c1/.5/-.6,   c2/.5/-1.2,
      d1/1.5/-.6
    }{
      \node[pf vertex] (\name) at (\x,\y) {};
    }

    \draw (r)--(a1)--(a2)--(a3)
          (r)--(b1)--(b2)
          (r)--(c1)--(c2)
          (r)--(d1);

    \foreach \x/\size in {-1.5/3,-.5/2,.5/2,1.5/1}{
      \node at (\x,-2.12) {\(\size\)};
    }

    \node at (0,-2.65) {\(P_{(3,2,2,1)}\)};
  \end{scope}

  % Non-root branching is excluded.
  \begin{scope}[xshift=4.6cm]
    \node[pf root,label=above:{root}] (s) at (0,0) {};
    \node[pf vertex] (v) at (0,-.6) {};
    \node[pf vertex] (u) at (-.5,-1.2) {};
    \node[pf vertex] (w) at (.5,-1.2) {};

    \draw (s)--(v)--(u) (v)--(w);
    \draw (v) circle[radius=3.4pt];

    \node[anchor=west,font=\scriptsize] at (.28,-.6)
      {two children};

    \node at (0,-2.65) {\(Y\) (not path-forest)};
  \end{scope}
\end{tikzpicture}

\caption{A path-forest tree \(P_{(3,2,2,1)}\) and the tree \(Y\).
Roots are at the top. The four numbers give the branch sizes, excluding the root.
The branches are unordered. The circled non-root vertex
of \(Y\) has two children.}
\label{fig:path-forest-trees}
\end{figure}
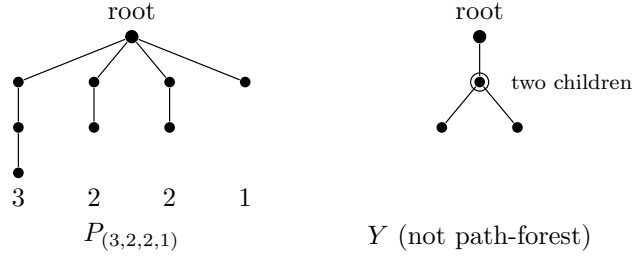

Let \(\Ipf\) be the span of the basis vectors \(F_t\) for which
\(t\) is not path-forest. Set
\(a_\lambda=\prod_{j\ge1}r_j(\lambda)!\), where \(r_j(\lambda)\) is
the multiplicity of the part \(j\). Hoffman's Hopf morphism
\(\mathsf Q:\GL\to\Lambda\) to the Hopf algebra of symmetric functions sends
\(F_{P_\lambda}\) to the augmented monomial
\(\widetilde m_\lambda=a_\lambda m_\lambda\) and vanishes on
\(\Ipf\) \cite[Equation~(4.2)]{hoffman2009rootedsymmetric}.
Here \(m_\lambda\) is the monomial symmetric function. The morphism induces
an isomorphism \(\GL/\Ipf\cong\Lambda\), and we prove that every
basis-positive character factors through this quotient.

Let
\[
 \King=\left\{p_1\ge p_2\ge\cdots\ge0:\ \sum_i p_i\le1\right\}
\]
carry the product topology. The missing mass
\(p_0=1-\sum_i p_i\) is called the dust mass. For a compact metrizable
space \(E\), we write \(\mathcal P(E)\) for its Borel probability
measures with the weak topology. With \(\mathsf p_k\) denoting the
power-sum symmetric functions, let \(\chi_p:\Lambda\to\R\) be the
unital algebra homomorphism determined by
\(\chi_p(\mathsf p_1)=1\) and
\(\chi_p(\mathsf p_k)=\sum_i p_i^k\) for \(k\ge2\).

\noindent\textbf{Theorem A (positive characters).}
Every basis-positive character of \(\GL\) vanishes on \(\Ipf\).
The normalized characters are uniquely parametrized by \(p\in\King\):
\(\psi_p=\chi_p\circ\mathsf Q\), with
\[
 \psi_p(F_{P_\lambda})=\chi_p(\widetilde m_\lambda),
 \qquad
 \psi_p(F_t)=0\quad\text{if }t\text{ is not path-forest}.
\]
The parametrization is a homeomorphism for pointwise convergence on the tree
basis.

More generally, apart from the zero map and the counit \(\varepsilon\),
every basis-positive character has the unique form
\[
 \psi(F_t)=c^{|t|-1}\psi_p(F_t),
 \qquad c=\psi(F_\tau)>0,\quad p\in\King.
\]
Here \(\varepsilon(F_t)=\mathbf1_{\{t=\one\}}\).

\medskip

For a positive left Pieri state, vanishing on \(\Ipf\) can be
tested through the right Pieri identity on chains. As the chain length
varies, the commutators contain trees with branching at every positive
depth. The left Pieri identity and positivity propagate their vanishing. 

\noindent\textbf{Theorem B (chain tests and the Kingman face).}
For a normalized positive left Pieri state \(L\), the chain identities
\[
 L(F_{C_m}F_\tau)=L(F_{C_m})\qquad(m\ge3)
\]
are equivalent to the full right Pieri identity
\(L(xF_\tau)=L(x)\) for every \(x\in\GL\), to the trace identity
\(L(xy)=L(yx)\) for every \(x,y\in\GL\), and to
\(L(\Ipf)=0\). These conditions hold if and only if there is a unique
\(\rho\in\mathcal P(\King)\) such that
\[
 L(x)=\int_{\King}\psi_p(x)\,\rho(dp)\qquad(x\in\GL).
\]
The central laws associated with these equivalent conditions are exactly
those supported on the path-forest subgraph. They form a proper exposed Bauer
face \(\mathscr C_{\mathrm{pf}}\), affinely homeomorphic to
\(\mathcal P(\King)\), whose extreme points are the laws
\(\mu^p:=\mu_{\psi_p}\). An explicit positive weighted sum of the chain
commutator values exposes this face.

Moreover, for each \(m\ge3\), there exists a normalized positive left Pieri state
satisfying all the other chain identities but not the identity at level \(m\).

\medskip
The path-forest subgraph is related to the classical Kingman graph by a
diagonal change of edge weights. We describe its Martin kernels in the
Hoffman normalization and identify its full and minimal Martin boundaries with
\(\King\). Under every law in \(\mathscr C_{\mathrm{pf}}\), the ranked
root-branch frequencies converge almost surely. Their limit is the paintbox
mixing parameter; it generates the completed central tail and determines the
unique decomposition into extreme central laws.

A unital basis-positive character \(\psi\) of arbitrary mass is
parametrized by its total mass \(c=\psi(F_\tau)\) and ranked atom masses
\(q_1\ge q_2\ge\cdots\ge0\), subject to \(\sum_iq_i\le c\).
Under Hopf convolution, these characters form a commutative monoid:
positive atom masses combine by multiset union and dust masses add.
In this monoid, an \(n\)-th convolution root, \(n\ge2\), exists if and
only if every positive atom mass has multiplicity divisible by \(n\),
and is then unique. The infinitely divisible elements are the pure-dust
characters, and every convolution semigroup has pure dust with mass
linear in time.

\medskip
\noindent\textbf{Background and organization.}
The rooted-tree algebra and growth multiplicities are due to
Grossman--Larson and Hoffman
\cite{grossman1989hopf,hoffman2003combinatorics}; see
\cite{oudom2008enveloping} for the enveloping-algebra construction and
\cite{zhao2008noncommutative,hoffman2009rootedsymmetric} for the related
symmetric-function morphisms. Kingman's paintbox representation and its
partition-probability form are developed in
\cite{kingman1978representation,pitman1995exchangeable,pitman2006combinatorial}.
The commutative boundary and its moment coordinates are described in
\cite{borodin2000harmonic,petrov2009kingman}, and the fixed-cotransition
approach in \cite{dynkin1978sufficient,vershik2015equipped}.

In a previous paper \cite{zhang2026recursivepaintboxesmartinboundary}, we determined the full
Martin boundary of the Hoffman graph in terms of recursive paintboxes and
obtained the canonical decomposition of all central laws. Here we identify
the Kingman face by algebraic conditions: basis-positive characters vanish
on non-path-forest trees, and two-sided Pieri invariance characterizes their
mixtures. The proofs use grafting identities and Kingman's paintbox representation. On
the path-forest subgraph, we give direct formulas for the Martin kernels
and recover the tail parameter from root-branch frequencies.

Section~\ref{sec:rooted} sets up the algebra and central-measure conventions.
Section~\ref{sec:path-forest-quotient} proves positivity rigidity, describes
Hoffman's quotient, and develops the chain tests and their independence.
Section~\ref{sec:kingman-boundary-characters} classifies positive
characters and two-sided Pieri states. Section~\ref{sec:kingman-face}
describes the exposed face and the Martin boundary,
Section~\ref{sec:central-tail} identifies the tail variable, and
Section~\ref{sec:finite-mass-convolution} treats finite mass and convolution.

\section{The rooted-tree algebra and central measures on the Hoffman graph}
\label{sec:rooted}

We use the rooted-tree conventions of
\cite[Section~2]{zhang2026recursivepaintboxesmartinboundary} and recall the central-measure
correspondence in the Hoffman normalization.

Throughout, \(\N=\{1,2,\ldots\}\), \([n]=\{1,\ldots,n\}\), and
\([0]=\varnothing\). We write \(\lambda\vdash N\) for an integer
partition of size \(|\lambda|=\sum_i\lambda_i=N\), with positive parts
in nonincreasing order and zero padding beyond its length. The empty
partition is \(\varnothing\), and \(a_\varnothing=1\).
The symbols \(\delta_x\), \(\delta_{ab}\), and \(\mathbf1_A\)
denote a point mass, the Kronecker delta, and an indicator, respectively.
Also, \(\mathcal L_\mu(Z)\) denotes the law of \(Z\) under \(\mu\), with
the subscript omitted when the underlying measure is understood.
For a compact space \(E\), \(C(E)\) denotes the real continuous
functions with the supremum norm.

\subsection{Rooted trees and the Hopf algebra}
\label{subsec:GL-product}

Let \(\T_n\) be the finite set of unlabelled non-plane rooted trees with
\(n\) vertices, and put \(\T=\bigsqcup_{n\ge1}\T_n\). For
\(t\in\T\), write \(|t|\) for its size, \(V(t)\) for the vertices of
a rooted representative, \(\Aut(t)\) for its rooted automorphism group,
\(\Leaves(t)\) for its leaves,
and \(t_v\) for the descendant subtree at \(v\), including \(v\).
The depth of a vertex is its distance from the root, measured in edges;
a branching vertex has at least two children.
Let \(C_m\) be the chain with \(m\) vertices, so
\(C_1=\one\) and \(C_2=\tau\).

A rooted forest is a finite unordered multiset of rooted trees. Deleting the
root of \(s\) gives a forest \(B_-(s)=s_1\cdots s_k\); the branches
are treated as distinct occurrences even when their shapes agree. Conversely,
\(B_+\) adjoins a new root to a forest, with
\(B_+(\varnothing)=\one\). Forest juxtaposition means unordered union.

The Grossman--Larson algebra \(\GL\) has basis \(F_t\), \(t\in\T\),
and product
\begin{equation}\label{eq:GL-product}
 F_sF_t=\sum_{f:[k]\to V(t)}F_{t\circ_f B_-(s)},
 \qquad B_-(s)=s_1\cdots s_k.
\end{equation}
Here \(t\circ_f B_-(s)\) is obtained by attaching the root of each
\(s_i\) to \(f(i)\). The resulting tree has the root of \(t\), and
all grafting targets belong to \(t\). Reordering branch occurrences or
changing rooted representatives induces bijections between the grafting
maps, so the coefficients are well-defined nonnegative integers. The product
is associative, has unit \(F_\one\), and is graded by
\(\deg F_t=|t|-1\)
\cite{grossman1989hopf,hoffman2003combinatorics}.

For \(t=B_+(t_1\cdots t_k)\) and \(I\subseteq[k]\), let
\(\mathfrak f_I=\prod_{i\in I}t_i\). The standard Hopf structure is
\begin{equation}\label{eq:GL-coproduct}
 \Delta(F_t)=\sum_{I\sqcup J=[k]}
 F_{B_+(\mathfrak f_I)}\otimes F_{B_+(\mathfrak f_J)},
 \qquad \varepsilon(F_t)=\mathbf1_{\{t=\one\}},
\end{equation}
where the bipartitions are ordered. This makes \(\GL\) connected,
graded, and cocommutative. In particular, its antipode \(S\) is determined
recursively in positive degree by
\begin{equation}\label{eq:GL-antipode-recursion}
 S(x)=-x-\sum S(x')x'',
 \qquad
 \Delta x=x\otimes F_\one+F_\one\otimes x+\sum x'\otimes x''.
\end{equation}
For linear functionals \(f,g:\GL\to\R\), their Hopf convolution is
\((f\star g)(x)=(f\otimes g)(\Delta x)\), with identity \(\varepsilon\).

A \emph{character} is a multiplicative linear functional
\(\psi:\GL\to\R\). The only nonunital character is the zero map:
\(\psi(F_\one)\) is an idempotent scalar and
\(\psi(x)=\psi(F_\one)\psi(x)\). A character is \emph{normalized}
when \(\psi(F_\tau)=1\), which implies unitality. The basis cone is
\[
 \GL_+=\left\{\sum_t a_tF_t:a_t\ge0,
                   \text{ with finite support}\right\}.
\]
A linear functional is basis-positive when it is nonnegative on this cone.
All positivity assertions below refer to this basis cone.

\subsection{Growth multiplicities and quotient histories}

For \(s\in\T\) and \(v\in V(s)\), let \(s^{+v}\) be obtained by
attaching a leaf to \(v\). Write \(s\nearrow t\) if
\(s^{+v}\cong t\) for some \(v\), and define
\[
 n(s;t)=\#\{v:s^{+v}\cong t\},\qquad
 m(s;t)=\#\{\ell\in\Leaves(t):t-\ell\cong s\}.
\]
The Hoffman graph \(\Gamma\) has edge multiplicities \(n(s;t)\).
A tree of size \(n\) is at graph level \(n-1\); endpoints and marginals
will always be indexed by tree size.

Let \(\SL(t)\) be the bijective labellings of a fixed representative of
\(t\) by \(\{0,\ldots,|t|-1\}\) that increase along every path away
from the root, and put
\(\QSL(t)=\SL(t)/\Aut(t)\). The automorphism action is free. If
\(d(t)=|\SL(t)|\) and \(u(t)=|\QSL(t)|\), the rooted-tree hook formula
and Hoffman's symmetry relation give
\begin{equation}\label{eq:rooted-hook}
 d(t)=\frac{|t|!}{\prod_{v\in V(t)}|t_v|},
 \qquad u(t)=\frac{d(t)}{|\Aut(t)|},
 \qquad |\Aut(s)|m(s;t)=n(s;t)|\Aut(t)|.
\end{equation}
For the hook formula, see
\cite{bergeron1992varieties,stanley2012enumerative}; for the symmetry
relation, see \cite[Proposition~2.1]{hoffman2003combinatorics}. The latter identity
also follows by counting rooted isomorphisms \(s^{+v}\to t\), first by
\(v\) and then by the image of the new leaf. Deleting the largest label
gives \(d(t)=\sum_{s\nearrow t}m(s;t)d(s)\), whence
\begin{equation}\label{eq:Hoffman-dimension-recursion}
 u(t)=\sum_{s\nearrow t}n(s;t)u(s),\qquad u(\one)=1.
\end{equation}
Thus \(u(t)>0\) is the weighted path dimension of \(t\), in agreement
with \cite[Proposition~2.6]{hoffman2003combinatorics}.

The single branch of \(\tau\) is a single vertex. Consequently,
\begin{equation}\label{eq:pieri-n}
 F_\tau F_t=\sum_{v\in V(t)}F_{t^{+v}}
           =\sum_{T:t\nearrow T}n(t;T)F_T.
\end{equation}
Iterating from the unit gives
\begin{equation}\label{eq:iterated-Pieri}
 F_\tau^m=\sum_{t\in\T_{m+1}}u(t)F_t,
 \qquad m\ge0.
\end{equation}

Replace each weighted edge by a set of \(n(s;t)\) edge copies. A compatible
identification of weighted paths with quotient labellings will be useful.

\begin{lemma}[Weighted paths and quotient labellings]
\label{lem:path-quotient-labelling-bijection}
Weighted paths from \(\one\) to \(t\) can be identified with
\(\QSL(t)\), compatibly with deletion of the largest label. Any two such compatible identifications are related by an
endpoint-preserving homeomorphism that preserves every central path law.
\end{lemma}

\begin{proof}
Fix a quotient labelling of \(s\) and one representative labelling \(L\).
For a cover \(s\nearrow t\), attach the new largest label at a vertex in
\(E(s,t)=\{v:s^{+v}\cong t\}\). The resulting quotient labellings of
\(t\) are distinct: any labelled isomorphism between two such extensions restricts,
after deletion of the new leaf, to an automorphism fixing \(L\), hence to
the identity. Conversely every quotient labelling whose restriction is the
chosen class arises this way. The restriction fibre therefore has
\(n(s;t)\) elements.

Choose a bijection from each such fibre to the fixed edge-copy set.
Successive deletion constructs the required compatible bijections. Two
choices differ by compatible endpoint-preserving permutations of the finite
paths, hence induce a homeomorphism of the infinite path space. A central
law assigns the same mass to all cylinders with a given endpoint, so these
permutations preserve that law.
\end{proof}

\subsection{Central measures and left Pieri states}
\label{subsec:central-measures}

Let \(\mathcal E_n\) be the finite set of edge copies from
\(\T_n\) to \(\T_{n+1}\), with source and range maps
\(\mathrm s,\mathrm r\). The weighted path space is
\[
 \mathcal X_\Gamma=
 \{(e_n)_{n\ge1}:\mathrm s(e_1)=\one,
                 \ \mathrm r(e_n)=\mathrm s(e_{n+1})\}.
\]
It is a closed subspace of \(\prod_{n\ge1}\mathcal E_n\), hence compact
metrizable. Write \(\mathsf E_n\) for the edge coordinates and
\(X_n\in\T_n\) for the endpoint coordinates. For a finite path
\(\gamma\) ending at \(t\in\T_n\), let \(C_\gamma\) be its prefix
cylinder. These clopen cylinders generate the Borel sigma-field and determine
weak convergence of probability measures, since their finite linear
combinations are uniformly dense in \(C(\mathcal X_\Gamma)\).

A probability measure \(\mu\) on \(\mathcal X_\Gamma\) is
\emph{central} when
\begin{equation}\label{eq:central-cylinder-followup}
 \mu(C_\gamma)=\frac{M_n^\mu(t)}{u(t)},
 \qquad M_n^\mu(t)=\mu\{X_n=t\}.
\end{equation}
Equivalently, prefixes are uniform conditional on their endpoints. Its
endpoint laws have canonical cotransitions
\begin{equation}\label{eq:canonical-cotransition-followup}
 p^\downarrow(T,t)=\frac{n(t;T)u(t)}{u(T)},
 \qquad
 M_n^\mu(t)=\sum_{T:t\nearrow T}M_{n+1}^\mu(T)p^\downarrow(T,t).
\end{equation}
A coherent family of endpoint laws conversely defines a unique central law
by \eqref{eq:central-cylinder-followup}, since coherence makes the
prefix-cylinder probabilities consistent. Thus central measures are
in affine bijection with normalized nonnegative harmonic functions
\begin{equation}\label{eq:harmonic-general}
 h(\one)=1,\qquad
 h(t)=\sum_{T:t\nearrow T}n(t;T)h(T),\qquad h(t)\ge0,
\end{equation}
through \(M_n(t)=u(t)h(t)\). Their set \(\mathscr C(\Gamma)\) is
compact convex: the central identities are closed affine conditions on
\(\mathcal P(\mathcal X_\Gamma)\). This is the fixed-cotransition
framework of \cite{dynkin1978sufficient,vershik2015equipped}.

\begin{definition}[Pieri states]
\label{def:Pieri-states}
A \emph{normalized positive left Pieri state} is a linear functional
\(L:\GL\to\R\) satisfying
\begin{equation}\label{eq:left-Pieri-state}
 L(F_\one)=1,\qquad L(F_t)\ge0\ (t\in\T),\qquad
 L(F_\tau x)=L(x)\ (x\in\GL).
\end{equation}
We call such a functional a \emph{left Pieri state} for short. It is
\emph{two-sided} if also \(L(xF_\tau)=L(x)\) for every \(x\in\GL\),
and is a \emph{trace} if \(L(xy)=L(yx)\) for every \(x,y\in\GL\).
\end{definition}

Equip the state space with pointwise convergence on the tree basis.
By \eqref{eq:pieri-n}, a left Pieri state is the linear extension of a
harmonic function \(h(t)=L(F_t)\). Hence these states are affinely
homeomorphic to \(\mathscr C(\Gamma)\), with
\begin{equation}\label{eq:left-Pieri-central-map}
 \mu_L(C_\gamma)=L(F_t),\qquad
 M_n^L(t)=u(t)L(F_t)\quad(t\in\T_n).
\end{equation}
Iterating the left Pieri identity gives \(L(F_\tau^m)=1\) for
\(m\ge0\). Equation \eqref{eq:iterated-Pieri} then yields
\(\sum_{t\in\T_{m+1}}u(t)L(F_t)=1\), hence
\(0\le L(F_t)\le u(t)^{-1}\) for every \(t\in\T\).
In particular, every normalized basis-positive character defines a
central measure.

\begin{remark}[Classification of left Pieri states]
\label{rem:left-Pieri-classification}
Let \(\partial_{\mathrm{RP}}\) be the compact space of deterministic
recursive paintboxes modulo equality of all finite rooted-tree sampling
laws, equipped with the topology of finite sampling laws. For
\(\xi\in\partial_{\mathrm{RP}}\), write \(M_n^\xi\) for its
\(n\)-vertex sampling law. Combining
\eqref{eq:left-Pieri-central-map} with
\cite[Theorem~A and Corollary~B]{zhang2026recursivepaintboxesmartinboundary} gives an
affine homeomorphism from \(\mathcal P(\partial_{\mathrm{RP}})\)
onto the space of left Pieri states:
\begin{equation}\label{eq:left-Pieri-boundary-mixture}
 L_\rho(F_t)
 =\int_{\partial_{\mathrm{RP}}}
   \frac{M_{|t|}^\xi(t)}{u(t)}\,\rho(d\xi),
 \qquad t\in\T.
\end{equation}
The representing measure \(\rho\) is unique. Thus the left Pieri
states form a Bauer simplex, with extreme points
\(L_{\delta_\xi}\), \(\xi\in\partial_{\mathrm{RP}}\).

The Kingman subboundary consists of recursive paintboxes with an
arbitrary root split and with each positive-mass root branch directed
by the boundary point whose \(n\)-vertex sampling law is
\(\delta_{C_n}\) for every \(n\ge1\).
Theorem~\ref{thm:two-sided-Pieri-characterization} identifies the
two-sided states with mixtures supported on this subboundary.
Extreme left Pieri states need not be characters; see
Proposition~\ref{prop:extremal-non-character}.
\end{remark}

\subsection{The central tail}

Let \(\mathscr F_n^0=\sigma(\mathsf E_1,\ldots,\mathsf E_{n-1})\)
be the sigma-field of prefixes up to size \(n\), with \(\mathscr F_1^0\)
trivial, and define the future and raw central tail fields by
\[
 \mathscr G_N^0=\sigma(X_N,\mathsf E_N,\mathsf E_{N+1},\ldots),
 \qquad \Tcen^0=\bigcap_{N\ge1}\mathscr G_N^0.
\]
For a law \(\mu\), let \(\mathcal N_\mu\) consist of all subsets
of \(\mu\)-null Borel sets, and write
\(\overline{\mathscr H}^{\,\mu}=\sigma(\mathscr H\cup\mathcal N_\mu)\)
for a Borel sub-sigma-field \(\mathscr H\). Set
\(\Tcen=\overline{\Tcen^0}^{\,\mu}\). A nonzero nonnegative harmonic
function \(h\) spans a \emph{minimal ray} if every harmonic \(g\) with
\(0\le g\le h\) is a scalar multiple of \(h\). The standard
fixed-cotransition extremality criterion \cite{dynkin1978sufficient}
takes the following form.

\begin{proposition}[Central extremality and the tail]
\label{prop:central-tail-extreme}
A central law is extreme in \(\mathscr C(\Gamma)\) if and only if its
raw central tail, equivalently its completed central tail, is trivial. Under
\eqref{eq:harmonic-general}, this is also equivalent to the associated
harmonic function spanning a minimal ray of the nonnegative harmonic cone.
\end{proposition}

\begin{proof}
If \(\gamma\) ends at \(t\in\T_n\), centrality and counting prefixes
give, for \(N\ge n\) and \(B\in\mathscr G_N^0\),
\begin{equation}\label{eq:prefix-uniformity-followup}
 \mu(C_\gamma\cap B)=\frac{\mu(\{X_n=t\}\cap B)}{u(t)}.
\end{equation}
It suffices to count for a fixed finite future segment and then use the
monotone-class theorem. Conditioning on a nontrivial raw-tail event therefore
splits \(\mu\) into two different central laws.

Conversely, suppose \(\mu=a\nu+(1-a)\nu'\), with \(0<a<1\) and
central \(\nu,\nu'\). For \(H=d\nu/d\mu\), prefix uniformity gives
\[
 \E_\mu[H\mid\mathscr F_n^0]
   =\frac{M_n^\nu(X_n)}{M_n^\mu(X_n)}=:H_n,
\]
with the ratio set to zero at zero-mass endpoints. The domination
\(\mu\ge a\nu\) gives \(0\le H_n\le a^{-1}\) on every path.
Since the prefix fields generate the Borel sigma-field, \(H_n\to H\)
almost surely and in \(L^1\). For each \(N\), all \(H_n\) with \(n\ge N\) are
\(\mathscr G_N^0\)-measurable; hence \(\limsup_nH_n\) is a raw-tail
version of \(H\). Tail triviality forces \(H=1\) and \(\nu=\mu\).
Completion does not change triviality. Finally, the affine harmonic
correspondence preserves extreme points; decomposing a normalized harmonic
function by a smaller nonnegative harmonic function gives the
minimal-ray equivalence.
\end{proof}

\section{Positivity rigidity and the path-forest quotient}
\label{sec:path-forest-quotient}

A tree is \emph{path-forest} if every non-root vertex has at most one child.
For an integer partition \(\lambda=(\lambda_1,\ldots,\lambda_k)\), let
\(P_\lambda\) have chain branches of these sizes. Thus
\(|P_\lambda|=1+|\lambda|\), \(P_\varnothing=\one\), and
\(P_{(k)}=C_{k+1}\) for \(k\ge1\). Put
\[
 \Ipf=\Span\{F_t:t\text{ is not path-forest}\}.
\]
We prove that basis-positive characters and two-sided Pieri states vanish
on \(\Ipf\), identify the symmetric-function quotient, and characterize
this vanishing by chain tests.

\subsection{Commutators and positivity propagation}

Let \(Y\) be the tree whose root has one child with two leaf children, and
let \(B=P_{(2,1)}\). The products in Figure~\ref{fig:GL-low-degree-products}
give the commutator identity used in the rigidity proof.

\begin{figure}[htbp]
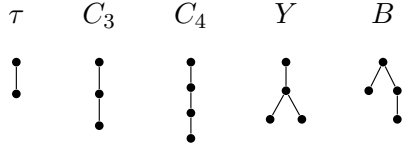

\centering
\begin{tabular}{c@{\qquad}c@{\qquad}c@{\qquad}c@{\qquad}c}
\(\tau\)&\(C_3\)&\(C_4\)&\(Y\)&\(B\)\\[.4em]
\TreeTauPic&\TreeCthreePic&\TreeCfourPic&\TreeYPic&\TreeBPic
\end{tabular}
\caption{The trees in the two products
\(F_\tau F_{C_3}=F_{C_4}+F_Y+F_B\) and
\(F_{C_3}F_\tau=F_{C_4}+F_B\). Roots are drawn at the top.}
\label{fig:GL-low-degree-products}
\end{figure}

\begin{lemma}[The four-vertex commutator]\label{lem:GL-kills-Y}
One has
\begin{equation}\label{eq:universal-Y-commutator}
 F_\tau F_{C_3}-F_{C_3}F_\tau=F_Y.
\end{equation}
Every real character and every two-sided Pieri state vanishes on \(F_Y\).
\end{lemma}

\begin{proof}
In \(F_\tau F_{C_3}\), grafting the single vertex at the root, middle
vertex, or leaf produces \(B,Y,C_4\), respectively. In
\(F_{C_3}F_\tau\), the chain branch can be grafted at either vertex of
\(\tau\), producing \(B\) or \(C_4\). Subtraction proves the identity.
A character annihilates commutators. For a two-sided Pieri state, both
products have value \(L(F_{C_3})\), so their difference has value zero.
\end{proof}

\begin{lemma}[Propagation of zeros]\label{lem:positivity-propagation}
Let \(L\) be either a basis-positive character or a two-sided
Pieri state. If \(L(F_A)=0\), then every basis vector occurring in
\(F_SF_A\) or \(F_AF_S\), for any tree \(S\), has value zero under
\(L\).
\end{lemma}

\begin{proof}
For a character, \(L(F_SF_A)=L(F_AF_S)=L(F_S)L(F_A)=0\).
Assume now that \(L\) is a two-sided Pieri state. Repeated application of
the two Pieri identities gives, for every \(m\ge0\),
\[
 L(F_\tau^mF_A)=L(F_A)=L(F_AF_\tau^m)=0.
\]
Expanding \(F_\tau^m\) by \eqref{eq:iterated-Pieri} yields
\begin{equation}\label{eq:two-sided-propagation-sums}
\begin{aligned}
 0&=\sum_{S\in\T_{m+1}}u(S)L(F_SF_A),\\
 0&=\sum_{S\in\T_{m+1}}u(S)L(F_AF_S).
\end{aligned}
\end{equation}
Every product lies in \(\GL_+\), so all summands are nonnegative. Since
\(u(S)>0\) for every \(S\), both product values vanish individually.
Allowing \(m\) to vary covers all trees \(S\).

Expanding either product in the tree basis and using positivity proves
the assertion.
\end{proof}

\begin{lemma}[Two-step grafting from \(Y\)]
\label{lem:constructive-propagation-from-Y}
Let \(T\) be a non-path-forest tree. There are rooted trees
\(A_0=Y,A_1,\ldots,A_m=T\), with \(0\le m\le2\), such that
for every \(j<m\), some rooted tree \(S_j\) has \(F_{A_{j+1}}\)
occurring with positive coefficient in \(F_{S_j}F_{A_j}\) or
\(F_{A_j}F_{S_j}\).
\end{lemma}

\begin{proof}
Choose a non-root branching vertex \(v\) of \(T\), let \(u\) be its
parent, and choose two children of \(v\). Retain \(u\), \(v\), and these
two children, together with their connecting edges. Rooted at \(u\), this
four-vertex subtree is a copy of \(Y\) inside \(T_u\).

First reconstruct \(T_u\). At each of the four retained vertices, collect
all child-subtree occurrences whose attachment edges were omitted. Keep each
such subtree in full, and denote this family by \(\mathcal O\). Set
\(S=B_+(\prod_{a\in\mathcal O}a)\). In \(F_SF_Y\), graft each
occurrence \(a\in\mathcal O\) to its original parent in the copy of
\(Y\). These subtrees are disjoint and account for every omitted vertex,
so this grafting produces \(T_u\). Thus \(F_{T_u}\) occurs with
positive coefficient in \(F_SF_Y\). If \(\mathcal O=\varnothing\),
then \(T_u=Y\) and this step is unnecessary.

If \(u\) is the root of \(T\), the construction is complete. Otherwise,
let \(R\) be obtained from \(T\) by deleting all strict descendants of
\(u\), leaving \(u\) as a leaf. In \(F_{T_u}F_R\), deleting the
root of the left factor leaves the child-subtree occurrences of \(u\).
Grafting all of them to the retained vertex \(u\) of \(R\)
reconstructs \(T\), so \(F_T\) occurs with positive coefficient.

Isomorphic branches are counted as distinct occurrences in both products.
This constructs the required sequence in at most two steps.
\end{proof}

\begin{theorem}[Positivity rigidity]\label{thm:GL-vanishing}
Every basis-positive Grossman--Larson character and every two-sided
Pieri state vanishes on \(\Ipf\).
\end{theorem}

\begin{proof}
Lemma~\ref{lem:GL-kills-Y} gives the initial zero at \(Y\).
For a non-path-forest tree \(T\), take the sequence from
Lemma~\ref{lem:constructive-propagation-from-Y}. Apply
Lemma~\ref{lem:positivity-propagation} at each of its at most two steps.
This gives \(L(F_T)=0\), and linearity gives \(L(\Ipf)=0\).
\end{proof}

\subsection{The symmetric-function quotient}

\begin{proposition}[The path-forest Hopf ideal]
\label{prop:path-forest-hopf-ideal}
The subspace \(\Ipf\) is a graded Hopf ideal.
\end{proposition}

\begin{proof}
If the right factor of a grafting product has a non-root branching vertex,
that vertex persists in every grafting term. If the left factor has such a
vertex, it lies in one of its root branches and remains non-root after
grafting. Thus every term is non-path-forest whenever either factor is,
proving the two-sided ideal property.

For a non-path-forest tree \(t=B_+(t_1\cdots t_k)\), some branch
\(t_j\) is not a chain. In each summand of \eqref{eq:GL-coproduct}, that
occurrence belongs to one of the two subforests. Its tensor factor is
therefore non-path-forest, so
\[
 \Delta\Ipf\subseteq\Ipf\otimes\GL+\GL\otimes\Ipf,
 \qquad \varepsilon(\Ipf)=0.
\]
The subspace is homogeneous. In each reduced term of
\eqref{eq:GL-antipode-recursion}, at least one factor belongs to \(\Ipf\)
and both have smaller degree. The ideal property and induction therefore
give \(S(\Ipf)\subseteq\Ipf\).
\end{proof}

Write \(\GLpf=\GL/\Ipf\) and \(\overline F_t\) for the image of
\(F_t\). We use the standard Hopf algebra
\(\Lambda=\R[\mathsf p_1,\mathsf p_2,\ldots]\) of symmetric
functions, where \(\mathsf p_k=\sum_i x_i^k\) has degree \(k\). Its coproduct is
evaluation on the disjoint union of two variable alphabets
\(\boldsymbol x\) and \(\boldsymbol y\):
\(\Delta f=f(\boldsymbol x\sqcup\boldsymbol y)\).
Let \(m_\lambda\) be
the monomial symmetric function and put
\[
 a_\lambda=\prod_{j\ge1}r_j(\lambda)!,\qquad
 \widetilde m_\lambda=a_\lambda m_\lambda,
 \qquad m_\varnothing=\widetilde m_\varnothing=1.
\]
See \cite{macdonald1995symmetric,stanley1999enumerative} for these
conventions.

The quotient map below is Hoffman's morphism
\cite[Equation~(4.2)]{hoffman2009rootedsymmetric}. Let
\(e_n=m_{(1^n)}\), where \((1^n)\) has \(n\) parts equal to one,
and put \(e_\lambda=\prod_i e_{\lambda_i}\). Hoffman's duality uses
\((e_\lambda,m_\eta)=\delta_{\lambda,\eta}\) and
\(\langle F_t,\mathfrak f\rangle
=|\Aut(t)|\mathbf1_{\{t\cong B_+(\mathfrak f)\}}\) for a rooted
forest \(\mathfrak f\)
\cite[Section~2 and Theorem~3.1]{hoffman2009rootedsymmetric}.
Under these pairings, the quotient morphism is dual to
\(e_n\mapsto C_n\), \(n\ge1\), into the Connes--Kreimer rooted-forest
Hopf algebra \cite{connes1998hopf,hoffman2003combinatorics}.
Since \(|\Aut(P_\lambda)|=a_\lambda\), the dual map has the
augmented monomial normalization above. The \(n\)-vertex chain has
degree \(n\) in the forest algebra, whereas \(F_{C_n}\) has
degree \(n-1\) in \(\GL\). We verify the quotient directly from grafting.

\begin{proposition}[Hoffman's path-forest quotient]
\label{prop:path-forest-hopf-quotient}
There is a graded Hopf algebra isomorphism
\(\Phi:\GLpf\to\Lambda\) satisfying
\(\Phi(\overline F_{P_\lambda})=\widetilde m_\lambda\) for every
integer partition \(\lambda\). In particular,
\(\Phi(\overline F_\tau)=\mathsf p_1\).
\end{proposition}

\begin{proof}
We first compare products. Distinguish the root-branch occurrences of
\(P_\lambda\) and \(P_\mu\), with sizes
\(\lambda_1,\ldots,\lambda_r\) and
\(\mu_1,\ldots,\mu_q\), respectively. A grafting term in
\(F_{P_\lambda}F_{P_\mu}\) survives modulo \(\Ipf\) if and only if
each incoming branch is grafted either to the root of \(P_\mu\), where
it remains separate, or to the leaf of a distinct branch, where the two
chains concatenate. Every other grafting creates non-root branching: it
either targets a non-root vertex with a child or sends two incoming branches to the
same leaf. Since all targets belong to \(P_\mu\), the surviving maps
correspond to partial matchings of the two sets of branch occurrences.

Such a map is encoded by a subset \(I\subseteq[r]\) and an injection
\(\alpha:I\hookrightarrow[q]\). For \(i\in I\), the parts
\(\lambda_i\) and \(\mu_{\alpha(i)}\) are replaced by their sum;
unmatched parts remain separate. Let \(\nu(I,\alpha)\) be the decreasing
rearrangement of the multiset
\[
    \{\lambda_i+\mu_{\alpha(i)}:i\in I\}
    \sqcup\{\lambda_i:i\notin I\}
    \sqcup\{\mu_j:j\notin\alpha(I)\}.
\]
Then
\begin{equation}\label{eq:path-forest-product-constants}
    \overline F_{P_\lambda}\,\overline F_{P_\mu}
    =\sum_{I,\alpha}\overline F_{P_{\nu(I,\alpha)}}.
\end{equation}
The augmented monomial basis has a corresponding formula. With the parts of
\(\lambda\) distinguished,
\begin{equation}\label{eq:augmented-monomial-injection-formula}
    \widetilde m_\lambda
    =\sum_{\iota:\{1,\ldots,r\}\hookrightarrow\N}
       \prod_{i=1}^r x_{\iota(i)}^{\lambda_i}.
\end{equation}
An ordinary monomial of type \(\lambda\) occurs once for each permutation
of equal parts, giving the factor \(\prod_j r_j(\lambda)!\). In the
product of two sums of the form
\eqref{eq:augmented-monomial-injection-formula}, the images of the two
injections may overlap. Injectivity makes this overlap a partial matching
between the two sets of parts. For a fixed matching \((I,\alpha)\),
merge each matched pair into one distinguished part whose exponent is the
sum, leaving the other parts distinguished. Assignments of these parts to
distinct variables are in bijection with the pairs of injections realizing
that matching. Its contribution is therefore
\(\widetilde m_{\nu(I,\alpha)}\), including the multiplicities of
equal parts. Summing over matchings gives
\(\widetilde m_\lambda\widetilde m_\mu
=\sum_{I,\alpha}\widetilde m_{\nu(I,\alpha)}\), as in
\eqref{eq:path-forest-product-constants}. Thus \(\Phi\) preserves products.
The classes \(\overline F_{P_\lambda}\) and the augmented monomials are
bases indexed by integer partitions, so \(\Phi\) is a graded algebra
isomorphism.

For the coproduct, \eqref{eq:GL-coproduct} splits the distinguished
root branches into two submultisets:
\begin{equation}\label{eq:path-forest-coproduct}
    \Delta(\overline F_{P_\lambda})
    =\sum_{I\sqcup J=[r]}
      \overline F_{P_{\lambda_I}}
      \otimes
      \overline F_{P_{\lambda_J}},
\end{equation}
where \(\lambda_I\) is the subpartition formed by the parts indexed by
\(I\), with \(\lambda_\varnothing=\varnothing\). Evaluating
\eqref{eq:augmented-monomial-injection-formula} on
\(\boldsymbol x\sqcup\boldsymbol y\) assigns each distinguished part
to one of these two alphabets, giving the same ordered splits
\(I\sqcup J=[r]\). Hence
\((\Phi\otimes\Phi)\Delta=\Delta\Phi\). The counits agree on the unit and
vanish in positive degree. Thus \(\Phi\) is a graded bialgebra isomorphism,
and therefore a Hopf algebra isomorphism.
\end{proof}

Thus the morphism used in the introduction is
\begin{equation}\label{eq:quotient-map-Q}
 \mathsf Q=\Phi\circ\pi_{\mathrm{pf}}:\GL\longrightarrow\Lambda,
 \qquad \pi_{\mathrm{pf}}:\GL\longrightarrow\GLpf,
\end{equation}
where \(\pi_{\mathrm{pf}}\) is the quotient map. In particular,
\(\ker\mathsf Q=\Ipf\) and \(\mathsf Q(F_\tau)=\mathsf p_1\).
Theorem~\ref{thm:GL-vanishing} says that every basis-positive character
factors through this map.

\begin{remark}[Comparison with Connes--Kreimer characters]
\label{rem:Connes-Kreimer-characters}
The Connes--Kreimer Hopf algebra \(H_{\mathrm{CK}}\) is the free
commutative algebra generated by rooted trees, with the empty forest
as unit \cite{connes1998hopf,hoffman2003combinatorics}.
Any family \((a_t)_{t\in\T}\) of real numbers therefore determines
a unique unital character through
\(\varphi_a(t_1\cdots t_r)=\prod_{j=1}^r a_{t_j}\).
This character is nonnegative on the forest basis if and only if
\(a_t\ge0\) for every \(t\), so positivity imposes no vanishing
condition on branching trees.

Under graded duality, unital characters of \(H_{\mathrm{CK}}\)
correspond to group-like elements in the degree completion of \(\GL\)
\cite[Proposition~4.4]{hoffman2003combinatorics}.
The present classification concerns multiplicative linear functionals on
\(\GL\), whose basis positivity is constrained by the grafting product.
\end{remark}

\subsection{Chain defects and independent tests}
\label{subsec:deterministic-counterexample}

For a left Pieri state, chain commutators test vanishing on \(\Ipf\).
We first prove this criterion and then show that the chain tests are
independent.

\begin{lemma}[Chain commutators and defect bounds]
\label{lem:chain-Pieri-defects}
For \(m\ge3\), label the vertices of \(C_m\) by
\(v_0,\ldots,v_{m-1}\) in increasing depth, and let \(T_{m,j}\)
be obtained by attaching one leaf at \(v_j\), where
\(1\le j\le m-2\). Thus \(|T_{m,j}|=m+1\), and its unique
branching vertex has depth \(j\). Then
\begin{equation}\label{eq:chain-positive-commutator}
 F_\tau F_{C_m}-F_{C_m}F_\tau
   =\sum_{j=1}^{m-2}F_{T_{m,j}}.
\end{equation}
For a left Pieri state \(L\), define its \emph{chain defects} by
\begin{equation}\label{eq:chain-Pieri-defect}
 \delta_m(L):=L(F_{C_m})-L(F_{C_m}F_\tau)
            =\sum_{j=1}^{m-2}L(F_{T_{m,j}}).
\end{equation}
They satisfy \(0\le\delta_m(L)\le1\). If a tree \(T\) has a
branching vertex at depth \(d\ge1\), then
\begin{equation}\label{eq:chain-defect-descendant-bound}
 0\le L(F_T)\le\delta_{d+2}(L).
\end{equation}
Consequently, all chain defects vanish if and only if \(L(\Ipf)=0\).
\end{lemma}

\begin{proof}
In \(F_\tau F_{C_m}\), attaching at the root or at the leaf
produces \(P_{(m-1,1)}\) or \(C_{m+1}\), respectively. These are
the two terms of \(F_{C_m}F_\tau\). Each other attachment gives one
\(T_{m,j}\). These trees have different branching depths, so each
occurs with coefficient one, proving \eqref{eq:chain-positive-commutator}.
The left Pieri identity yields \eqref{eq:chain-Pieri-defect} and
nonnegativity. Since \(F_{C_m}F_\tau\in\GL_+\), one also has
\(\delta_m(L)\le L(F_{C_m})\le1\); the last inequality follows
from \eqref{eq:iterated-Pieri}, \(u(C_m)=1\), and
\(L(F_\tau^{m-1})=1\).

Choose a branching vertex \(v\) of \(T\) at depth \(d\). Keep
the path from the root to \(v\) and two children of \(v\), with all
other vertices deleted. The resulting rooted subtree has \(d+3\) vertices
and is \(U=T_{d+2,d}\). It contains every ancestor of each retained
vertex. Adding the omitted vertices in an order with every parent before
its children therefore gives a leaf-growth sequence
\(U=s_0\nearrow s_1\nearrow\cdots\nearrow s_\ell=T\).
For each step, \eqref{eq:harmonic-general} and positivity give
\[
 L(F_{s_j})=\sum_{R:s_j\nearrow R}n(s_j;R)L(F_R)
       \ge n(s_j;s_{j+1})L(F_{s_{j+1}}).
\]
Since every edge multiplicity is a positive integer, iteration yields
\[
 \delta_{d+2}(L)\ge L(F_U)
   \ge\left(\prod_{j=0}^{\ell-1}n(s_j;s_{j+1})\right)L(F_T)
   \ge L(F_T).
\]
The first inequality uses the occurrence of \(F_U\) in
\eqref{eq:chain-Pieri-defect}; for \(U=T\), the product is empty
and equals one. This proves \eqref{eq:chain-defect-descendant-bound}.
Every non-path-forest tree has such a vertex, so vanishing of all defects
implies \(L(\Ipf)=0\). Conversely, each \(T_{m,j}\) is
non-path-forest, so \(L(\Ipf)=0\) makes every defect vanish.
\end{proof}

To separate the chain tests, fix an integer \(h\ge1\) and define a
leaf-growth path \((t_n^{(h)})_{n\ge1}\) as follows. Up to size
\(h+1\), set
\(t_n^{(h)}=C_n\). Thereafter keep a stem of length \(h\) from
the root to a vertex \(v\), and attach all further vertices as leaves
to \(v\). Thus, for \(n\ge h+1\), the vertex at depth \(h\)
has \(n-h-1\) leaf children. The first non-path-forest endpoint is
\(t_{h+3}^{(h)}=T_{h+2,h}\).

\begin{proposition}[Independent chain tests]
\label{prop:extremal-non-character}
For every \(h\ge1\), the Dirac law \(\mu_h\) on the path
\((t_n^{(h)})_{n\ge1}\) is central and extreme.
Its left Pieri state \(L_h\) satisfies
\begin{equation}\label{eq:independent-chain-defects}
 \delta_m(L_h)=\mathbf1_{\{m=h+2\}},\qquad m\ge3.
\end{equation}
Thus each chain identity can fail while all the others hold, and
\(\mu_1\) is not induced by a Grossman--Larson character.
\end{proposition}

\begin{proof}
The chain endpoints have dimension one. For \(n\ge h+1\), put
\(r=n-h-1\). The stem vertices have hook sizes
\(n,n-1,\ldots,r+1\), and the \(r\) attached leaves have hook size
one. The hook formula and the permutations of these leaves give
\[
 \prod_{x\in V(t_n^{(h)})}|(t_n^{(h)})_x|=\frac{n!}{r!},
 \qquad d(t_n^{(h)})=|\Aut(t_n^{(h)})|=r!,
 \qquad u(t_n^{(h)})=1.
\]
The compatible endpoint sequence therefore has a unique weighted prefix at
every size. Its Dirac law \(\mu_h\) is central and extreme, and
\eqref{eq:left-Pieri-central-map} gives
\(L_h(F_t)=\mathbf1_{\{t=t_{|t|}^{(h)}\}}\).

Each \(T_{m,j}\) has a single branching vertex with two children.
The endpoint \(t_{m+1}^{(h)}\) is a chain when \(m<h+2\),
equals \(T_{h+2,h}\) when \(m=h+2\), and has at least three
children at its branching vertex when \(m>h+2\). At \(m=h+2\),
the branching depth singles out the term \(j=h\). Summing its
values in \eqref{eq:chain-Pieri-defect} proves
\eqref{eq:independent-chain-defects}. Finally,
\(t_4^{(1)}=Y\), so \(L_1(F_Y)=1\), whereas every character
vanishes on \(F_Y\) by Lemma~\ref{lem:GL-kills-Y}.
\end{proof}

\section{Paintbox characters and two-sided Pieri states}
\label{sec:kingman-boundary-characters}

The quotient reduces the classification to positive Pieri functionals on
symmetric functions. Kingman's representation expresses them as paintbox
mixtures; multiplicativity forces the mixing measure to be a point mass.

\subsection{Paintbox coordinates}
\label{subsec:paintbox-coordinates}

Recall the ranked simplex
\(\King=\{p_1\ge p_2\ge\cdots\ge0:\sum_i p_i\le1\}\), with the
product topology and dust mass \(p_0=1-\sum_i p_i\). For \(k\ge2\),
put \(S_k(p)=\sum_i p_i^k\). The paintbox specialization is the
unital algebra homomorphism \(\chi_p:\Lambda\to\R\) determined by
\[
 \chi_p(\mathsf p_1)=1,\qquad
 \chi_p(\mathsf p_k)=S_k(p)\quad(k\ge2).
\]

Associate to \(p\) the mass-weighted probability measure
\(\nu_p=p_0\delta_0+\sum_{i\ge1}p_i\delta_{p_i}\) on \([0,1]\).
Its moments satisfy \(\int x^r\,\nu_p(dx)=S_{r+1}(p)\) for
\(r\ge1\). These are the standard Kingman moment coordinates, with
an index shift from \cite[Section~2.2]{petrov2009kingman}.
We record the continuity and separation properties used below.

\begin{lemma}[Continuity and identification of paintbox coordinates]
\label{lem:Kingman-simplex-continuity}
The space \(\King\) is compact metrizable. The functions \(S_k\),
\(k\ge2\), are continuous and separate its points. In particular,
\(p\mapsto\chi_p(f)\) is continuous for every \(f\in\Lambda\).
\end{lemma}

\begin{proof}
The ordering and finite partial-sum inequalities define a closed subset of
\([0,1]^{\N}\). For every \(p\in\King\),
\begin{equation}\label{eq:Kingman-uniform-moment-tail}
 \sum_{i>L}p_i^k
 \le p_{L+1}^{k-1}\sum_{i>L}p_i
 \le (L+1)^{1-k},\qquad k\ge2.
\end{equation}
Thus the power sums are uniform limits of continuous functions.

If \(S_k(p)=S_k(q)\) for all \(k\ge2\), then the probability
measures \(\nu_p\) and \(\nu_q\) have equal moments of every
nonnegative integer order. Polynomial approximation on \([0,1]\)
gives \(\nu_p=\nu_q\). For \(x>0\), its multiplicity in the
ranked mass sequence is \(\nu_p(\{x\})/x\), and the dust mass is
\(\nu_p(\{0\})\). Hence \(p=q\). Every symmetric function is a
polynomial in finitely many power sums, which also proves continuity of
\(p\mapsto\chi_p(f)\).
\end{proof}

Let \(\mathfrak P_\infty\) be the space of partitions of \(\N\),
viewed as the closed subspace of equivalence relations in
\(\{0,1\}^{\N\times\N}\), and let \(\Pi\) denote its canonical
partition. Write \(\mathfrak P_n\) for the partitions of \([n]\), with
\(\mathfrak P_0\) containing only the empty partition. A partition law
is \emph{exchangeable} if it is invariant under all permutations of
\(\N\) with finite support.
The paintbox directed by \(p\) assigns each positive integer
label independently to atom \(i\) with probability \(p_i\), or to dust
with probability \(p_0\). Labels assigned to the same positive atom form a
block; each dust label forms a new singleton. Write \(Q^p\) for the law of
this partition of \(\N\). For a finite set partition \(\pi\), let
\(\lambda(\pi)\) be its ranked block-size partition.

\begin{lemma}[Augmented monomials as partition probabilities]
\label{lem:paintbox-augmented-monomial-probabilities}
For a fixed partition \(\pi\) of \([n]\),
\begin{equation}\label{eq:paintbox-fixed-partition}
 Q^p\{\Pi|_{[n]}=\pi\}=\chi_p(\widetilde m_{\lambda(\pi)}).
\end{equation}
In particular, these specializations are nonnegative on the augmented
monomial basis.
\end{lemma}

\begin{proof}
Write the block sizes as \(\lambda_1,\ldots,\lambda_\ell\), treating
the blocks as distinguished, and let \(J=\{a:\lambda_a=1\}\). Only
singleton blocks may come from dust. Choosing the subset \(D\subseteq J\)
of such blocks and assigning the remaining blocks to distinct positive atoms
gives
\begin{equation}\label{eq:paintbox-explicit-dust}
 Q^p\{\Pi|_{[n]}=\pi\}
 =\sum_{D\subseteq J}p_0^{|D|}
   \sum_{\iota:[\ell]\setminus D\hookrightarrow\N}\,
       \prod_{a\notin D}p_{\iota(a)}^{\lambda_a}.
\end{equation}
We set empty products and \(p_0^0\) equal to one; the unique empty
injection contributes one.
When \(p_0=0\), the right side of
\eqref{eq:paintbox-explicit-dust} reduces to
\[
 \sum_{\iota:[\ell]\hookrightarrow\N}
       \prod_{a=1}^{\ell}p_{\iota(a)}^{\lambda_a}.
\]
This is the evaluation of \(\widetilde m_\lambda\) by
\eqref{eq:augmented-monomial-injection-formula}. The series converges since
it is bounded by \(\prod_a\sum_i p_i^{\lambda_a}\le1\); its equality
with symmetric-function evaluation follows by truncating the variables.

For general \(p\), replace dust by \(M\) additional atoms of mass
\(p_0/M\), leaving the original atoms unchanged. The resulting ordinary
evaluation has first power sum one and, for \(k\ge2\), power sum
\[
 S_k(p)+M(p_0/M)^k\longrightarrow S_k(p).
\]
Thus its value on \(\widetilde m_\lambda\) tends to
\(\chi_p(\widetilde m_\lambda)\). Couple the partitions by assigning
each dust label independently to one of the \(M\) new atoms. On \([n]\)
they can differ only when two labels choose the same new atom, an event of
probability at most \(\binom n2 p_0^2/M\). Passing to the limit proves
\eqref{eq:paintbox-fixed-partition}.
\end{proof}

For ordinary monomials, let \(\lambda\setminus1^k\)
denote the partition obtained by deleting \(k\) singleton parts. Grouping
\eqref{eq:paintbox-explicit-dust} by \(k=|D|\) and using
\(\binom{r_1(\lambda)}k a_{\lambda\setminus1^k}/a_\lambda=1/k!\)
gives
\begin{equation}\label{eq:paintbox-ordinary-monomial-dust}
 \chi_p(m_\lambda)
 =\sum_{k=0}^{r_1(\lambda)}\frac{p_0^k}{k!}\,
       m_{\lambda\setminus1^k}(p_1,p_2,\ldots).
\end{equation}
The terms on the right are ordinary evaluations on the atom sequence,
whose first power sum is \(1-p_0\). Formula
\eqref{eq:paintbox-ordinary-monomial-dust} is the Kingman boundary
kernel in the monomial normalization of
\cite[Equation~(4.8)]{borodin2000harmonic}.

The finite-partition cylinders are clopen and determine weak convergence on
\(\mathfrak P_\infty\). Hence \eqref{eq:paintbox-fixed-partition}
and Lemma~\ref{lem:Kingman-simplex-continuity} show that
\(p\mapsto Q^p\) is weakly continuous and hence defines a Borel probability
kernel.

\subsection{Positive Pieri functionals on symmetric functions}

The following identities express normalization and restriction consistency
of partition probabilities in the augmented monomial basis
\cite{pitman1995exchangeable,borodin2000harmonic}:
\begin{align}
 \mathsf p_1^n
   &=\sum_{\pi\in\mathfrak P_n}\widetilde m_{\lambda(\pi)},
       \label{eq:set-partition-expansion-p1-power}\\
 \mathsf p_1\widetilde m_{\lambda(\pi)}
   &=\sum_{\substack{\pi'\in\mathfrak P_{n+1}\\
                     \pi'|_{[n]}=\pi}}
       \widetilde m_{\lambda(\pi')}.
       \label{eq:set-partition-augmented-pieri}
\end{align}
For the first identity, expand \((\sum_jx_j)^n\) and group maps
\([n]\to\N\) by their kernel partitions. For the second, multiplying
the injection formula by \(\sum_jx_j\) either places the new label in one
of the existing distinguished blocks or gives it a new variable and hence a
singleton block. Equal block sizes are counted separately before collecting
integer-partition types.

Let \(\mathscr S_{\mathrm{Pieri}}\) be the set of linear functionals
\(\ell:\Lambda\to\R\) satisfying
\begin{equation}\label{eq:Pieri-state-conditions}
 \ell(1)=1,\qquad \ell(\widetilde m_\lambda)\ge0,
 \qquad \ell(\mathsf p_1f)=\ell(f)\quad(f\in\Lambda).
\end{equation}
We give this space the topology of pointwise convergence. A compact convex
set is a \emph{Bauer simplex} if its extreme points form a closed set and
each point is the barycentre of a unique probability measure on that set.

\begin{proposition}[Paintbox representation of positive Pieri functionals]
\label{prop:path-forest-Pieri-state-simplex}
The map
\begin{equation}\label{eq:Pieri-state-barycentre}
 \rho\longmapsto\ell_\rho,\qquad
 \ell_\rho(f)=\int_{\King}\chi_p(f)\,\rho(dp),
\end{equation}
is an affine homeomorphism from \(\mathcal P(\King)\) onto
\(\mathscr S_{\mathrm{Pieri}}\). In particular, this space is a Bauer
simplex, with extreme points \(\chi_p\), \(p\in\King\).
\end{proposition}

\begin{proof}
Every \(\ell_\rho\) satisfies \eqref{eq:Pieri-state-conditions} by
Lemma~\ref{lem:paintbox-augmented-monomial-probabilities} and
\(\chi_p(\mathsf p_1)=1\). Conversely, let
\(\ell\in\mathscr S_{\mathrm{Pieri}}\). Assign to each
\(\pi\in\mathfrak P_n\) the weight
\(w_n(\pi)=\ell(\widetilde m_{\lambda(\pi)})\). These weights are
nonnegative and depend only on the block sizes. The two identities above
give, respectively,
\[
 \sum_{\pi\in\mathfrak P_n}w_n(\pi)=\ell(\mathsf p_1^n)=1,
 \qquad
 \sum_{\pi':\,\pi'|_{[n]}=\pi}w_{n+1}(\pi')=w_n(\pi).
\]
They therefore define an exchangeable random partition of \(\N\).
Kingman's representation theorem
\cite{kingman1978representation,pitman2006combinatorial} gives a unique
\(\rho\in\mathcal P(\King)\) such that its law is
\(\int Q^p\,\rho(dp)\). By
\eqref{eq:paintbox-fixed-partition},
\[
 \ell(\widetilde m_\lambda)
   =\int_{\King}\chi_p(\widetilde m_\lambda)\,\rho(dp).
\]
The augmented monomials are a basis, so \(\ell=\ell_\rho\).

Continuity follows from Lemma~\ref{lem:Kingman-simplex-continuity}.
The domain is compact and the pointwise topology is Hausdorff, so this affine
bijection is a homeomorphism. Since the extreme points of
\(\mathcal P(\King)\) are its point masses, and they form a closed set,
the extreme-point description follows.
\end{proof}

\subsection{Classification of characters and two-sided states}

\begin{theorem}[Normalized basis-positive characters]
\label{thm:positive-GL-characters}
The normalized basis-positive characters of \(\GL\) are
\(\psi_p=\chi_p\circ\mathsf Q\), \(p\in\King\). Equivalently,
\[
 \psi_p(F_{P_\lambda})=\chi_p(\widetilde m_\lambda),\qquad
 \psi_p(F_t)=0\quad\text{if }t\text{ is not path-forest}.
\]
The parameter is unique, and \(p\mapsto\psi_p\) is a homeomorphism onto
the character space with pointwise convergence on the tree basis.
\end{theorem}

\begin{proof}
The displayed maps are multiplicative because \(\mathsf Q\) and
\(\chi_p\) are algebra morphisms. They are normalized and basis-positive
by Lemma~\ref{lem:paintbox-augmented-monomial-probabilities}.
Conversely, by Theorem~\ref{thm:GL-vanishing}, any normalized basis-positive
character \(\psi\) factors through \(\mathsf Q\), giving a character
\(\chi\) with \(\chi(1)=\chi(\mathsf p_1)=1\) and
\(\chi(\widetilde m_\lambda)=\psi(F_{P_\lambda})\ge0\).
Multiplicativity gives \(\chi(\mathsf p_1f)=\chi(f)\), so
\(\chi\in\mathscr S_{\mathrm{Pieri}}\). Write
\(\chi=\ell_\rho\) by Proposition~\ref{prop:path-forest-Pieri-state-simplex}.
For each \(k\ge2\), multiplicativity implies
\begin{equation}\label{eq:character-mixing-variance}
 \int S_k(p)^2\,\rho(dp)=\chi(\mathsf p_k^2)
 =\chi(\mathsf p_k)^2
 =\left(\int S_k(p)\,\rho(dp)\right)^2.
\end{equation}
Each \(S_k\) is therefore constant almost surely. On a common set of
\(\rho\)-measure one all these countably many functions are constant;
by Lemma~\ref{lem:Kingman-simplex-continuity} this set has only one
point. Thus \(\rho=\delta_p\), proving the classification and uniqueness.
The coordinate functions of \(p\mapsto\psi_p\) are continuous, and a
continuous bijection from compact \(\King\) to the Hausdorff character
space is a homeomorphism.
\end{proof}

\begin{theorem}[Two-sided Pieri characterization]
\label{thm:two-sided-Pieri-characterization}
Let \(L:\GL\to\R\) be a normalized positive left Pieri state. The
following are equivalent:
\begin{enumerate}[label=(\roman*)]
\item \(L(F_{C_m}F_\tau)=L(F_{C_m})\) for every \(m\ge3\);
\item \(L(xF_\tau)=L(x)\) for every \(x\in\GL\);
\item \(L(xy)=L(yx)\) for every \(x,y\in\GL\);
\item \(L(\Ipf)=0\);
\item there is a unique \(\rho\in\mathcal P(\King)\) such that
\begin{equation}\label{eq:two-sided-Pieri-mixture}
 L(x)=\int_{\King}\psi_p(x)\,\rho(dp)\qquad(x\in\GL).
\end{equation}
\end{enumerate}
Moreover, for each \(m\ge3\), there exists a normalized positive left Pieri
state for which the identity in \textup{(i)} fails at level \(m\) and holds
at every other level. The mixture map is an
affine homeomorphism onto the two-sided states, whose extreme points are the
normalized basis-positive characters.
\end{theorem}

\begin{proof}
Condition (i) says that \(\delta_m(L)=0\) for every \(m\ge3\).
Lemma~\ref{lem:chain-Pieri-defects} then gives (iv).

If (iv) holds, \(L\) factors uniquely as \(\ell\circ\mathsf Q\).
The quotient is commutative, so for \(x,y\in\GL\),
\[
 L(xy)=\ell(\mathsf Q(x)\mathsf Q(y))
      =\ell(\mathsf Q(y)\mathsf Q(x))=L(yx).
\]
This proves (iii). Condition (iii) and the left Pieri identity give
\(L(xF_\tau)=L(F_\tau x)=L(x)\), proving (ii), which implies (i)
by restriction to chains. The cases \(C_1=\one\) and
\(C_2=\tau\) already satisfy the right identity by left invariance.
For each \(m\ge3\),
Proposition~\ref{prop:extremal-non-character} provides a normalized positive
left Pieri state satisfying every chain identity except the one at level
\(m\).

For a functional satisfying (iv), its quotient \(\ell\) satisfies
\(\ell(1)=L(F_\one)=1\) and
\(\ell(\widetilde m_\lambda)=L(F_{P_\lambda})\ge0\).
For the Pieri identity, choose \(x\) with \(\mathsf Q(x)=f\) and use
\[
 \ell(\mathsf p_1f)=L(F_\tau x)=L(x)=\ell(f).
\]
Thus \(\ell\in\mathscr S_{\mathrm{Pieri}}\), and
Proposition~\ref{prop:path-forest-Pieri-state-simplex} gives the unique
representation (v). Conversely, every mixture in (v) satisfies (iv), since
all \(\psi_p\) vanish on \(\Ipf\). The topology of the mixture
space and its extreme points follow from the same proposition under the
identification \(L=\ell\circ\mathsf Q\).
\end{proof}

\begin{example}[A nonmultiplicative mixture]
\label{ex:two-sided-nonmultiplicative}
Let \(p^{\mathrm d}=(0,0,\ldots)\) be the pure-dust parameter and
\(p^{\mathrm a}=(1,0,0,\ldots)\) the one-atom parameter. Their
central laws are concentrated, respectively, on the star path
\(P_{(1^N)}\) and the chain path \(C_{N+1}\), \(N\ge0\). The functional
\(L=\tfrac12(\psi_{p^{\mathrm d}}+\psi_{p^{\mathrm a}})\)
is a two-sided Pieri state and a trace. Since
\(\psi_{p^{\mathrm d}}(F_{C_3})=0\) and
\(\psi_{p^{\mathrm a}}(F_{C_3})=1\), it satisfies
\[
 L(F_{C_3}^2)=\tfrac12
       \ne\tfrac14=L(F_{C_3})^2.
\]
Thus a two-sided state need not be a character.
\end{example}

\begin{remark}[Traces and the left Pieri condition]
\label{rem:traces-left-Pieri}
Even for unital basis-positive traces, the normalization
\(L(F_\tau)=1\) does not imply the left Pieri condition in
Theorem~\ref{thm:two-sided-Pieri-characterization}.
Fix \(p\in\King\). For \(c>0\), define
\(\psi_p^{(c)}(F_t)=c^{|t|-1}\psi_p(F_t)\).
The grading makes each \(\psi_p^{(c)}\) a unital basis-positive
character. Hence
\(L_0=\tfrac12(\psi_p^{(1/2)}+\psi_p^{(3/2)})\)
is a basis-positive trace satisfying
\(L_0(F_\one)=L_0(F_\tau)=1\), whereas
\(L_0(F_\tau^2)=5/4\).
Thus \(L_0\) is not a left Pieri state.
\end{remark}

\subsection{The common kernel and uniform completion}

We now determine the common kernel of the normalized characters and identify
the resulting quotient with the moment-coordinate algebra
\cite[Sections~2.2--2.3]{petrov2009kingman}. Put
\[
 \mathcal J_1=\Ipf+\langle F_\tau-F_\one\rangle_{\mathrm{2s}},
\]
where the brackets denote the generated two-sided algebra ideal.

\begin{proposition}[The normalized-character kernel]
\label{prop:normalized-positive-character-kernel-completion}
One has
\begin{equation}\label{eq:normalized-positive-character-kernel}
 \bigcap_{p\in\King}\ker\psi_p=\mathcal J_1,
 \qquad \GL/\mathcal J_1\cong\Lambda/(\mathsf p_1-1).
\end{equation}
Evaluation embeds this quotient as \(\R[S_2,S_3,\ldots]\), uniformly
dense in \(C(\King)\). Its completion for
\(\|[x]\|_+=\sup_{p\in\King}|\psi_p(x)|\) is therefore
\(C(\King)\).
\end{proposition}

\begin{proof}
We first show that the \(\chi_p\) separate
\(\Lambda/(\mathsf p_1-1)\). Write a class as a polynomial
\(g(\mathsf p_2,\ldots,\mathsf p_d)\). For \(d\ge2\), take
\(x_1>\cdots>x_{d-1}>0\) with \(\sum_{j<d}x_j<1\). These
inequalities define a nonempty open set in \(\R^{d-1}\). For the parameter
\(p=(x_1,\ldots,x_{d-1},0,\ldots)\), the specialization satisfies
\(\chi_p(\mathsf p_1)=1\), and its higher moments are given by
\[
 (x_1,\ldots,x_{d-1})\longmapsto
 \left(\sum_{j<d}x_j^2,\ldots,\sum_{j<d}x_j^d\right).
\]
Its Jacobian \(J\), with rows indexed by \(2\le k\le d\) and
columns by \(1\le j<d\), satisfies
\begin{equation}\label{eq:moment-map-Jacobian}
 \begin{aligned}
 J_{k,j}&=kx_j^{k-1},\\
 \det J&=d!\left(\prod_{j=1}^{d-1}x_j\right)
               \prod_{1\le i<j\le d-1}(x_j-x_i)\ne0.
 \end{aligned}
\end{equation}
Factoring \(k\) from each row and \(x_j\) from each column leaves
Vandermonde rows of degrees \(0,\ldots,d-2\).
The inverse function theorem applied to \eqref{eq:moment-map-Jacobian}
gives an open set of moment values. A polynomial vanishing at all paintbox
moments vanishes on this open set and is zero. The case of a constant
polynomial is immediate. Hence the common kernel on \(\Lambda\) is
\((\mathsf p_1-1)\).

Its inverse image under \(\mathsf Q\) is \(\mathcal J_1\). Indeed,
if \(\mathsf Q(x)=(\mathsf p_1-1)f\), choose \(y\) with
\(\mathsf Q(y)=f\). Then \(x-(F_\tau-F_\one)y\in\Ipf\).
This proves \eqref{eq:normalized-positive-character-kernel} and injectivity
of evaluation. The image is \(\R[S_2,S_3,\ldots]\); it contains the
constants and separates points by
Lemma~\ref{lem:Kingman-simplex-continuity}. Stone--Weierstrass gives
the stated density and completion.
\end{proof}

\section{The Kingman face and its Martin boundary}
\label{sec:kingman-face}

The block-to-chain correspondence realizes the classified states as central
laws supported on path-forest trees. We identify their exposed face and
describe the Martin boundary of the path-forest subgraph through
finite-population sampling.

\subsection{Central paintbox laws and the exposed face}

Let \(\Gamma_{\mathrm{pf}}\) be the subgraph of path-forest trees, with
the inherited edge multiplicities. Its level \(N\) consists of
\(P_\lambda\), \(\lambda\vdash N\). Its path space is the closed
subset \(\mathcal X_{\Gamma_{\mathrm{pf}}}\) of
\(\mathcal X_\Gamma\) consisting of paths for which \(X_n\) is path-forest
for every \(n\ge1\). Once non-root branching appears, leaf growth cannot
remove it. Hence every path ending at a path-forest tree has stayed
in this subgraph. Its path dimensions and cotransitions therefore agree
with those inherited from \(\Gamma\).

Fix a compatible path--quotient-labelling identification as in
Lemma~\ref{lem:path-quotient-labelling-bijection}.

\begin{lemma}[The block-to-chain correspondence]
\label{lem:block-chain-correspondence}
Weighted paths to \(P_\lambda\), \(\lambda\vdash N\), correspond to
partitions of \([N]\) of type \(\lambda\), compatibly with restriction.
The resulting homeomorphism
\(F:\mathfrak P_\infty\to\mathcal X_{\Gamma_{\mathrm{pf}}}\)
identifies exchangeable partition laws with central path laws. In particular,
\begin{equation}\label{eq:path-forest-dimension-formula}
 u(P_\lambda)=\frac{N!}{\prod_i\lambda_i!\,a_\lambda}.
\end{equation}
\end{lemma}

\begin{proof}
For a quotient labelling of \(P_\lambda\), group the positive labels by
root branch. Conversely, arrange each block in its unique increasing chain
order and attach that chain to the root labelled zero. The two maps are
inverse and commute with restriction to an initial label set. Both infinite
maps are consequently continuous.
Counting partitions of type \(\lambda\) gives
\eqref{eq:path-forest-dimension-formula}.

The symmetric group on \([N]\) is transitive on partitions of a fixed
type. Thus exchangeability is equivalent to uniformity within each type,
which under the bijection is centrality. The choice of compatible edge-copy
identification does not change the pushed-forward laws by
Lemma~\ref{lem:path-quotient-labelling-bijection}.
\end{proof}

For \(p\in\King\), denote by \(\mu^p\) the central law of
\(\psi_p\). The preceding correspondence realizes it as \(F_*Q^p\):
each finite path ending at \(P_\lambda\) has probability
\(\chi_p(\widetilde m_\lambda)=\psi_p(F_{P_\lambda})\). Hence
\begin{equation}\label{eq:paintbox-central-endpoints}
 M_n^p(P_\lambda)=u(P_\lambda)\chi_p(\widetilde m_\lambda),
 \qquad |\lambda|=n-1.
\end{equation}

\begin{theorem}[The exposed Kingman face]
\label{thm:path-forest-Bauer-face}
Let
\[
 \mathscr C_{\mathrm{pf}}
 =\{\mu\in\mathscr C(\Gamma):
      \mu(X_n\text{ is path-forest})=1\text{ for every }n\}.
\]
This is a proper exposed face, with exposing functional
\begin{equation}\label{eq:exposing-functional}
 \mathscr D(\mu_L)=\sum_{m\ge3}2^{-m}\delta_m(L).
\end{equation}
The barycentre map
\begin{equation}\label{eq:Kingman-barycentre-map}
 \mathsf B:\mathcal P(\King)\longrightarrow\mathscr C_{\mathrm{pf}},
 \qquad \mathsf B(\rho)=\int_{\King}\mu^p\,\rho(dp),
\end{equation}
is an affine homeomorphism, so the face is a Bauer simplex with extreme
points \(\mu^p\), \(p\in\King\). These laws are also extreme in
\(\mathscr C(\Gamma)\).
\end{theorem}

\begin{proof}
An exposed face is the set of minimizers of a continuous affine functional.
Each \(\delta_m(L)\) is a finite linear combination of tree
coordinates, so it is continuous and affine under the correspondence
\eqref{eq:left-Pieri-central-map}. The bounds
\(0\le\delta_m(L)\le1\) in Lemma~\ref{lem:chain-Pieri-defects}
give uniform convergence of \eqref{eq:exposing-functional}. Hence
\(\mathscr D\) is continuous, affine, and nonnegative.

Its zero set consists of the states whose chain defects all vanish.
By Theorem~\ref{thm:two-sided-Pieri-characterization}, these are the
states annihilating \(\Ipf\), which correspond to the central laws in
\(\mathscr C_{\mathrm{pf}}\). Thus \(\mathscr D\) exposes this
face. For the extreme laws of
Proposition~\ref{prop:extremal-non-character},
\(\mathscr D(\mu_h)=2^{-(h+2)}>0\), proving that the face is proper.

The same characterization identifies \(\mathscr C_{\mathrm{pf}}\) with the
two-sided Pieri states, equivalently with the traces among left Pieri states,
and yields the unique mixture representation. Since
\(p\mapsto\mu^p=F_*Q^p\) is weakly continuous, it defines a Borel probability
kernel, so \eqref{eq:Kingman-barycentre-map} is well defined. The affine
homeomorphism follows by composing the state parametrization in
Theorem~\ref{thm:two-sided-Pieri-characterization} with
\eqref{eq:left-Pieri-central-map}. Its extreme points are the images of
the point masses and form a closed set. An extreme point of a face is
extreme in the containing convex set, so each \(\mu^p\) is also extreme
in \(\mathscr C(\Gamma)\).
\end{proof}

Restriction and zero extension identify central laws on
\(\Gamma_{\mathrm{pf}}\) with \(\mathscr C_{\mathrm{pf}}\): the
inherited path counts give the same central cylinder equations. Since the
subgraph path space is closed, this identification is an affine
homeomorphism for weak convergence.

\subsection{Kingman normalization and finite-population kernels}

Let \(\mathbb Y\) be the set of all integer partitions. For
\(\mu\vdash r\) and \(\lambda\vdash N\), write
\(n_{\mathrm{pf}}(\mu;\lambda)\) for the weighted number of paths from
\(P_\mu\) to \(P_\lambda\), with the value zero when \(r>N\)
and \(n_{\mathrm{pf}}(\mu;\lambda)=\delta_{\mu\lambda}\) when
\(r=N\). We use the following normalization of the Martin kernel:
\begin{equation}\label{eq:path-forest-Martin-kernel}
 K^{\mathrm{pf}}_\mu(\lambda)
   =\frac{u(P_\mu)n_{\mathrm{pf}}(\mu;\lambda)}{u(P_\lambda)}.
\end{equation}
For \(r\le N\), these coordinates sum to one over \(\mu\vdash r\).

The classical Kingman graph \(\mathbb K\) has the same partition vertices.
If \(\lambda\) is obtained from \(\mu\) by adding a singleton part
or by increasing a part of size \(k-1\) to \(k\), its edge weight is
\(\kappa_{\mathbb K}(\mu,\lambda)=r_k(\lambda)\), with \(k=1\)
in the singleton case \cite[Section~4]{borodin2000harmonic}.
The Hoffman weights satisfy
\begin{equation}\label{eq:path-forest-Kingman-gauge}
 n(P_\mu;P_\lambda)
    =\frac{a_\mu}{a_\lambda}\kappa_{\mathbb K}(\mu,\lambda).
\end{equation}
For a new singleton, the left side is one and
\(a_\lambda/a_\mu=r_1(\lambda)\). For an extension with \(k\ge2\),
the left side is \(r_{k-1}(\mu)\) and
\[
 \frac{a_\lambda}{a_\mu}
      =\frac{r_k(\mu)+1}{r_{k-1}(\mu)},
 \qquad r_k(\lambda)=r_k(\mu)+1.
\]
This proves \eqref{eq:path-forest-Kingman-gauge}. Let
\(\kappa_{\mathbb K}(\mu;\lambda)\) be the corresponding weighted
number of paths, including the length-zero path, and set
\(\dim_{\mathbb K}(\lambda)=\kappa_{\mathbb K}(\varnothing;\lambda)\).
The ratios in \eqref{eq:path-forest-Kingman-gauge} telescope along paths, so
\[
 n_{\mathrm{pf}}(\mu;\lambda)
   =\frac{a_\mu}{a_\lambda}\kappa_{\mathbb K}(\mu;\lambda),
 \qquad
 \dim_{\mathbb K}(\lambda)=a_\lambda u(P_\lambda)
                         =\frac{N!}{\prod_i\lambda_i!}.
\]
Consequently the normalized Martin kernels agree:
\[
 K^{\mathrm{pf}}_\mu(\lambda)
 =\frac{\dim_{\mathbb K}(\mu)\kappa_{\mathbb K}(\mu;\lambda)}
        {\dim_{\mathbb K}(\lambda)}.
\]
The augmented monomial normalization and the change of graph weights thus
use the same factor \(a_\lambda\).

\begin{lemma}[A uniform finite-population estimate]
\label{lem:path-forest-partition-kernel}
Let \(\Pi_\lambda\) be uniform among the partitions of \([N]\) of
type \(\lambda\), where \(N\ge1\), and set \(q_i=\lambda_i/N\).
For \(1\le r\le N\),
\begin{equation}\label{eq:finite-population-TV}
 \left\|\mathcal L(\Pi_\lambda|_{[r]})
          -\mathcal L_{Q^q}(\Pi|_{[r]})\right\|_{\mathrm{TV}}
 \le 1-\frac{(N)_r}{N^r}
 \le\frac{\binom r2}{N},
\end{equation}
where \((N)_r=N(N-1)\cdots(N-r+1)\) is the falling factorial and
\(\|P-Q\|_{\mathrm{TV}}=\sup_A|P(A)-Q(A)|\). Moreover,
\begin{equation}\label{eq:path-forest-Martin-partition-restriction}
 K^{\mathrm{pf}}_\mu(\lambda)
   =\Prob\{\lambda(\Pi_\lambda|_{[r]})=\mu\},
 \qquad \mu\vdash r.
\end{equation}
Thus, if \(N_j\to\infty\), \(\lambda^{(j)}\vdash N_j\), and
\(\lambda_i^{(j)}/N_j\to p_i\) for every fixed \(i\), then
\(p\in\King\) and
\begin{equation}\label{eq:path-forest-kernel-paintbox-limit}
 K^{\mathrm{pf}}_\mu(\lambda^{(j)})
   \longrightarrow u(P_\mu)\chi_p(\widetilde m_\mu)
 \qquad(\mu\in\mathbb Y).
\end{equation}
\end{lemma}

\begin{proof}
Divide a population of \(N\) individuals into distinguished cells of
sizes \(\lambda_i\). For each partition of type \(\lambda\), there
are \(a_\lambda\prod_i\lambda_i!\) bijections from \([N]\) to the
population that induce it: equally sized blocks can be assigned to cells
in \(a_\lambda\) ways, and their labels can be assigned to individuals
in \(\prod_i\lambda_i!\) ways. A uniform bijection therefore gives
\(\Pi_\lambda\). Restricting to \([r]\) samples \(r\) individuals
without replacement and records their cells.
With replacement, the cells are independent with probabilities \(q_i\),
so they give the paintbox \(Q^q\). Conditional on having no repeated
individual, the ordered with-replacement sample is uniform among distinct
\(r\)-tuples. Keep this sample on that event and otherwise draw a fresh
uniform distinct tuple; the resulting tuple has the without-replacement law.
The two samples therefore agree with probability at least \((N)_r/N^r\).
Recording cells cannot increase the discrepancy probability. A union bound
over pairs proves the second inequality in \eqref{eq:finite-population-TV}.
This is the sampling-without-replacement comparison used in finite
exchangeability \cite{diaconis1980finite}.

For a partition \(\pi\in\mathfrak P_r\) with distinguished block
sizes \(\mu_1,\ldots,\mu_\ell\), counting ordered samples gives
\begin{equation}\label{eq:finite-population-fixed-partition}
 \Prob\{\Pi_\lambda|_{[r]}=\pi\}
 =\frac1{(N)_r}
   \sum_{\iota:[\ell]\hookrightarrow\N}\,
       \prod_{a=1}^{\ell}(\lambda_{\iota(a)})_{\mu_a}.
\end{equation}
Here \((b)_0=1\), \((b)_j=b(b-1)\cdots(b-j+1)\) for
\(1\le j\le b\), and \((b)_j=0\) for \(j>b\), with
\(b\) a nonnegative integer. The formula assigns distinct population cells
to distinct blocks and distinct individuals to their labels.

Under Lemma~\ref{lem:block-chain-correspondence}, a uniform partition of
type \(\lambda\) is a uniform weighted path to \(P_\lambda\).
There are \(u(P_\mu)n_{\mathrm{pf}}(\mu;\lambda)\) such paths through
\(P_\mu\), proving
\eqref{eq:path-forest-Martin-partition-restriction}. There are
\(u(P_\mu)\) partitions of type \(\mu\). Combining \eqref{eq:finite-population-fixed-partition} with the
dimension formula gives the following kernel for \(\mu\vdash r\) and
\(r\le N\):
\begin{equation}\label{eq:finite-population-explicit-kernel}
 K^{\mathrm{pf}}_\mu(\lambda)
 =\frac{r!}{a_\mu\prod_{a=1}^{\ell}\mu_a!}\,
   \frac{1}{(N)_r}
   \sum_{\iota:[\ell]\hookrightarrow\N}
      \prod_{a=1}^{\ell}(\lambda_{\iota(a)})_{\mu_a}.
\end{equation}
This is the factorial-monomial form of the Kingman kernel
\cite[Section~3]{petrov2009kingman}, with the dimension factor
\(u(P_\mu)\) giving probability-valued coordinates.
The same type event under \(Q^q\) has probability
\(u(P_\mu)\chi_q(\widetilde m_\mu)\). Applying
\eqref{eq:finite-population-TV} to this event gives
\begin{equation}\label{eq:finite-population-kernel-error}
 \left|K^{\mathrm{pf}}_\mu(\lambda)
       -u(P_\mu)\chi_q(\widetilde m_\mu)\right|
 \le \frac{\binom r2}{N}.
\end{equation}
The normalized partitions belong to the closed space \(\King\), so their
coordinatewise limit \(p\) also belongs to \(\King\). For each fixed
\(\mu\vdash r\), eventually \(N_j\ge r\), and the right side of
\eqref{eq:finite-population-kernel-error} tends to zero. Continuity from
Lemma~\ref{lem:Kingman-simplex-continuity} proves
\eqref{eq:path-forest-kernel-paintbox-limit}. Formula
\eqref{eq:finite-population-explicit-kernel} also holds for \(r=0\)
with the empty-product conventions; the empty-partition coordinate is one.
\end{proof}

\subsection{The full and minimal Martin boundaries}

We use the fixed-cotransition Martin compactification
\cite[Sections~2.3 and~3.3]{vershik2015equipped}. Embed \(\mathbb Y\) in
\([0,1]^{\mathbb Y}\) by
\(\lambda\mapsto(K^{\mathrm{pf}}_\mu(\lambda))_\mu\), and denote its
closure by \(\Omega_M(\Gamma_{\mathrm{pf}})\). This space is compact
metrizable because \(\mathbb Y\) is countable. We also write
\(K^{\mathrm{pf}}_\mu\) for the continuous coordinate function on
this closure. The Martin boundary
\(\partial\Omega_M(\Gamma_{\mathrm{pf}})\) is the closure minus the
finite vertices. For a boundary point \(z\), put
\(h_z(P_\mu)=K^{\mathrm{pf}}_\mu(z)/u(P_\mu)\). The minimal boundary
consists of the points for which \(h_z\) is normalized nonnegative
harmonic and spans a minimal ray.

\begin{theorem}[Kingman boundary and convergence criterion]
\label{thm:path-forest-Kingman-boundary}
The full and minimal Martin boundaries of \(\Gamma_{\mathrm{pf}}\)
coincide and are homeomorphic to \(\King\) via
\begin{equation}\label{eq:path-forest-boundary-kernel}
 K^{\mathrm{pf}}_\mu(p)=u(P_\mu)\chi_p(\widetilde m_\mu).
\end{equation}
For any sequence \(N_j\to\infty\) and
\(\lambda^{(j)}\vdash N_j\), convergence to the boundary point \(p\)
is equivalent to \(\lambda_i^{(j)}/N_j\to p_i\) for every fixed \(i\).
\end{theorem}

\begin{proof}
First, finite vertices are isolated in the compactification. For
\(\lambda\vdash N\), the conditions
\[
 K^{\mathrm{pf}}_\lambda(z)>\tfrac12,
 \qquad \sum_{\nu\vdash N+1}K^{\mathrm{pf}}_\nu(z)<\tfrac12
\]
define an open set \(U_\lambda\) containing its embedded point and no
other finite vertex: the first condition excludes lower levels and other
vertices at level \(N\), and the second excludes all higher levels.
The set \(U_\lambda\setminus\{\lambda\}\) is open and disjoint from
the dense set of finite vertices, so it is empty. The embedding is also
injective by the same coordinate identities.

By metrizability, every boundary point is a limit of finite vertices.
Their levels tend to infinity, since each bounded set of levels consists
of finitely many isolated vertices. Write these vertices as
\(\lambda^{(j)}\vdash N_j\), with \(N_j\to\infty\). By compactness of
\(\King\), some subsequence of
\((\lambda_i^{(j)}/N_j)_i\) converges coordinatewise to a parameter
\(p\). Lemma~\ref{lem:path-forest-partition-kernel} identifies the kernel
limit along this subsequence with \eqref{eq:path-forest-boundary-kernel}.
Thus every boundary point has the stated form.

Conversely, fix \(p\in\King\). Let \(b_N=\lfloor\sqrt N\rfloor\),
take parts \(\lfloor Np_i\rfloor\) for \(1\le i\le b_N\), discard
zeros, and add
\(N-\sum_{i\le b_N}\lfloor Np_i\rfloor\) singleton parts. Their ranked
union is a partition \(\lambda^{(N)}\) of \(N\). For every fixed
\(i\), once \(b_N\ge i\),
\[
 \frac{\lfloor Np_i\rfloor}{N}
 \le\frac{\lambda_i^{(N)}}N
 \le\max\left\{\frac{\lfloor Np_i\rfloor}{N},\frac1N\right\}.
\]
Hence its normalized coordinates converge to \(p\), and its kernels
converge to \eqref{eq:path-forest-boundary-kernel}. For every fixed partition level, 
the limiting kernel coordinates sum to one.
No finite vertex has this property at all levels, so the limit belongs to the
Martin boundary.

For a one-part partition \((k)\), the chain \(P_{(k)}=C_{k+1}\)
has \(u(P_{(k)})=1\), and \(\widetilde m_{(k)}=\mathsf p_k\).
Thus its boundary coordinate is \(S_k(p)\) for \(k\ge2\).
These functions separate points. Continuity of all coordinates and
compactness give the homeomorphism.

For \(p\in\King\), \eqref{eq:path-forest-boundary-kernel} identifies
the boundary function as \(h_p(P_\mu)=\psi_p(F_{P_\mu})\).
The left Pieri identity and vanishing off the subgraph make \(h_p\)
normalized nonnegative harmonic on \(\Gamma_{\mathrm{pf}}\), with
central law the restriction of \(\mu^p\). This law is extreme by
Theorem~\ref{thm:path-forest-Bauer-face} and the restriction
correspondence above. Proposition~\ref{prop:central-tail-extreme},
applied to \(\Gamma_{\mathrm{pf}}\), therefore shows that \(h_p\)
spans a minimal ray.

Coordinatewise convergence of the normalized partitions implies kernel
convergence by \eqref{eq:finite-population-kernel-error}. Conversely, suppose
the kernels converge to those of \(p\), and set
\(q^{(j)}=(\lambda_i^{(j)}/N_j)_{i\ge1}\).
For fixed \(k\ge2\) and all large \(j\),
\eqref{eq:finite-population-explicit-kernel} and
\eqref{eq:finite-population-kernel-error} give
\begin{equation}\label{eq:one-part-kernel-moment-control}
 \begin{aligned}
 K^{\mathrm{pf}}_{(k)}(\lambda^{(j)})
   &=\frac{\sum_i(\lambda_i^{(j)})_k}{(N_j)_k},\\
 |S_k(q^{(j)})-S_k(p)|
   &\le \frac{\binom k2}{N_j}
       +|K^{\mathrm{pf}}_{(k)}(\lambda^{(j)})-S_k(p)|
       \longrightarrow0.
 \end{aligned}
\end{equation}
Every subsequential limit of \(q^{(j)}\) in the compact space
\(\King\) therefore has the same power sums as \(p\), by continuity
and \eqref{eq:one-part-kernel-moment-control}.
Lemma~\ref{lem:Kingman-simplex-continuity} identifies that limit
with \(p\). Compactness gives convergence of the whole sequence.
\end{proof}

In the recursive-paintbox boundary of \cite{zhang2026recursivepaintboxesmartinboundary}, the
point \(p\) is represented by the root mass partition \(p\), with dust
mass \(p_0\), and a single atom of mass one, with no dust, at every
descendant split. Each occupied positive root block then grows as a chain,
while root dust produces singleton branches. Its finite sampling laws are
\eqref{eq:paintbox-central-endpoints}. Since these laws determine the full
boundary point and depend continuously on \(p\), this construction
identifies the Kingman boundary above with a closed subspace of the full
Hoffman boundary.

\begin{corollary}[The mass-weighted profile]
\label{cor:mass-profile-continuity}
The mass-weighted profile \(p\mapsto\nu_p\) defined in
Subsection~\ref{subsec:paintbox-coordinates} is continuous for weak
convergence on \([0,1]\).
In particular, under the convergence criterion above,
\[
 \sum_i\frac{\lambda_i^{(j)}}{N_j}\,
       \delta_{\lambda_i^{(j)}/N_j}\Longrightarrow\nu_p.
\]
\end{corollary}

\begin{proof}
All these measures have total mass one, and their moments of order
\(r\ge1\) are \(S_{r+1}(p)\). Continuity of these moments and polynomial
approximation on \([0,1]\) prove the assertion.
\end{proof}

\section{Root frequencies and the central tail}\label{sec:central-tail}

Under a central law in the Kingman face, the boundary parameter can be read
from the asymptotic sizes of the root branches. We show that this parameter
generates the central tail and determines the conditional path law.

For a path-forest tree \(t=P_\lambda\), write \(\lambda(t):=\lambda\) for
its root-branch partition, with parts listed in nonincreasing order. Thus, for
\(P_\lambda\in\T_n\), one has \(|\lambda|=n-1\). Define
\begin{equation}\label{eq:empirical-root-frequencies}
    \mathbf P_n(P_\lambda)
    :=\left(\frac{\lambda_1}{n-1},
             \frac{\lambda_2}{n-1},\ldots\right)
    \in\King,
    \qquad n\ge2,
\end{equation}
with zero padding. Define also the singleton-branch density and the
mass-weighted empirical branch measure:
\begin{equation}\label{eq:empirical-dust-and-mass-measure}
    D_n(P_\lambda):=\frac{r_1(\lambda)}{n-1},
    \qquad
    \nu_n(P_\lambda)
    :=\sum_{j\ge1}\frac{\lambda_j}{n-1}
       \delta_{\lambda_j/(n-1)}.
\end{equation}
The measure \(\nu_n(P_\lambda)\) is a probability measure on \([0,1]\).
Set \(\mathbf P_1=\mathbf 0=(0,0,\ldots)\), \(D_1=0\), and
\(\nu_1=\delta_0\). On non-path-forest states, use the same values,
so all three observables are defined on every Hoffman state.

The block-to-chain correspondence turns root-branch sizes into occupancy
counts. For the dust-free frequency limit in the classical occupancy scheme,
see \cite[Proposition~26]{gnedin2007occupancy}.

\begin{lemma}[Empirical branch frequencies, dust, and mass profile]
\label{lem:deterministic-paintbox-frequency-limit}
For every \(p\in\King\), under the central measure \(\mu^p\), almost surely,
\[
    \mathbf P_n(X_n)\longrightarrow p
    \quad\text{coordinatewise},
    \qquad
    D_n(X_n)\longrightarrow p_0,
    \qquad
    \nu_n(X_n)\Longrightarrow\nu_p.
\]
\end{lemma}

\begin{proof}
Use the block-to-chain representation of \(\mu^p\). Let \(N_i(n)\)
count labels among \(1,\ldots,n-1\) assigned to positive atom \(i\),
and let \(Z_0(n)\) count dust labels. Apply the strong law to the
indicators of individual atoms, the sets of atoms with index greater than
\(L\), and dust. On a common probability-one event, for every \(i,L\ge1\),
\[
 \frac{N_i(n)}{n-1}\longrightarrow p_i,\qquad
 \frac{\sum_{i>L}N_i(n)}{n-1}\longrightarrow\sum_{i>L}p_i,\qquad
 \frac{Z_0(n)}{n-1}\longrightarrow p_0.
\]
Work on this event. Fix \(k\ge1\) and \(\varepsilon>0\), and choose
\(L\ge k\) with \(\sum_{i>L}p_i<\varepsilon\). Eventually every
tail-atom block has normalized size at most \(2\varepsilon\), since
this bounds their total mass, and each dust block has normalized size
\(1/(n-1)<\varepsilon\). Let \((b_1^{(n)},\ldots,b_L^{(n)})\) be the decreasing rearrangement
of the first \(L\) normalized counts. Monotonicity of order statistics
under coordinatewise comparison gives
\[
 \max_{1\le j\le L}|b_j^{(n)}-p_j|
 \le \max_{1\le i\le L}\left|\frac{N_i(n)}{n-1}-p_i\right|
 \longrightarrow0.
\]
Let \(a_k^{(n)}=[\mathbf P_n(X_n)]_k\). Including all tail-atom
and dust blocks gives
\[
 b_k^{(n)}\le a_k^{(n)}\le\max\{b_k^{(n)},2\varepsilon\}.
\]
Let \(n\to\infty\) and then \(\varepsilon\downarrow0\). This proves
coordinatewise convergence to \(p\) on the common event.

The number of singleton blocks contributed by positive atoms is
\(R_n=\sum_{i\ge1}\mathbf1_{\{N_i(n)=1\}}\), and
\(r_1(\lambda(X_n))=Z_0(n)+R_n\). For every \(L\),
\[
 0\le\frac{R_n}{n-1}
 \le\frac{L}{n-1}+\frac{\sum_{i>L}N_i(n)}{n-1},
 \qquad
 \limsup_{n\to\infty}\frac{R_n}{n-1}\le\sum_{i>L}p_i.
\]
Sending \(L\to\infty\) gives \(R_n/(n-1)\to0\), and hence
\(D_n(X_n)\to p_0\). Finally,
\(\nu_n(X_n)=\nu_{\mathbf P_n(X_n)}\) for \(n\ge2\), because
these empirical frequency vectors have total mass one.
Corollary~\ref{cor:mass-profile-continuity} proves the profile limit.
\end{proof}

\begin{theorem}[The paintbox variable and the central tail]
\label{thm:path-forest-tail}
There is a \(\Tcen^0\)-measurable map
\(\mathbf P^0:\mathcal X_\Gamma\to\King\) with the following
properties under every \(\mu\in\mathscr C_{\mathrm{pf}}\).
Write \(\mathbf P_\mu=(\mathbf P_{\mu,i})_{i\ge1}\) for this map
viewed under \(\mu\), and put
\(\mathbf P_{\mu,0}=1-\sum_{i\ge1}\mathbf P_{\mu,i}\).
Almost surely,
\begin{equation}\label{eq:tail-empirical-limits}
 \mathbf P_n(X_n)\longrightarrow\mathbf P_\mu,
 \qquad D_n(X_n)\longrightarrow\mathbf P_{\mu,0},
 \qquad \nu_n(X_n)\Longrightarrow\nu_{\mathbf P_\mu},
\end{equation}
where the first convergence is coordinatewise.
The endpoints \(X_n\) also converge almost surely to \(\mathbf P_\mu\)
in the path-forest Martin compactification.

The completed central tail and conditional path law satisfy
\begin{equation}\label{eq:conditional-tail-law}
 \Tcen=\overline{\sigma(\mathbf P_\mu)}^{\,\mu},
 \qquad
 \Prob_\mu(A\mid\mathbf P_\mu)=\mu^{\mathbf P_\mu}(A)
 \quad\mu\text{-a.s.},
\end{equation}
for every Borel path event \(A\).
With \(\Theta_\mu=\mathcal L_\mu(\mathbf P_\mu)\),
the representation
\(\mu=\int_{\King}\mu^p\,\Theta_\mu(dp)\)
is the unique decomposition of \(\mu\) into extreme points of
\(\mathscr C(\Gamma)\).
In particular, \(\mu\) is extreme if and only if
\(\mathbf P_\mu\) is almost surely constant.
\end{theorem}

\begin{proof}
\emph{Step 1: Construction of a common tail map.}
For each path, set
\[
 a_i=\limsup_{n\to\infty}[\mathbf P_n(X_n)]_i,\qquad
 E=\bigcap_{i\ge1}
   \left\{\liminf_{n\to\infty}[\mathbf P_n(X_n)]_i=a_i\right\}.
\]
Since the coordinates lie in \([0,1]\), \(E\) is the event of
coordinatewise convergence.
For every \(N\), all terms with \(n\ge N\) are
\(\mathscr G_N^0\)-measurable. Since removing finitely many terms changes
neither convergence nor the limsup, \(a_i\) and \(E\) are measurable
with respect to \(\Tcen^0\). On \(E\), the vector \(\mathbf a=(a_i)_i\)
belongs to the closed space \(\King\). Define
\begin{equation}\label{eq:common-path-forest-tail-map}
 \mathbf P^0=
 \begin{cases}
 \mathbf a,&\text{on }E,\\
 (0,0,\ldots),&\text{on }E^c.
 \end{cases}
\end{equation}
This \(\Tcen^0\)-measurable map is independent of the choice of law.

\medskip\noindent
\emph{Step 2: Empirical limits and the conditional law.}
Fix \(\mu\in\mathscr C_{\mathrm{pf}}\) and write
\(\mu=\int\mu^p\,\rho(dp)\) by
Theorem~\ref{thm:path-forest-Bauer-face}. The preceding lemma gives
\begin{equation}\label{eq:common-tail-map-fibre-support}
 \mu^p\{\mathbf P^0=p\}=1\qquad(p\in\King),
\end{equation}
as well as the dust and profile limits. The convergence events are Borel
in \((p,\omega)\): coordinate convergence is a countable condition,
\(p\mapsto p_0\) is Borel, and weak convergence on \([0,1]\) is
determined by integer moments. Their sections have \(\mu^p\)-probability
one for every \(p\), so the limits hold almost surely under the joint law
\(\rho(dp)\mu^p(d\omega)\). Substituting
\(p=\mathbf P^0(\omega)\) gives \eqref{eq:tail-empirical-limits}
under the path marginal \(\mu\). The Martin convergence follows from
Theorem~\ref{thm:path-forest-Kingman-boundary}.

For a Borel path event \(A\) and a Borel \(B\subseteq\King\), the
fibre identity \eqref{eq:common-tail-map-fibre-support} gives
\begin{equation}\label{eq:disintegration-fibre-identity}
 \begin{aligned}
 \mu(A\cap\{\mathbf P_\mu\in B\})
   &=\int_{\King}\mu^p(A\cap\{\mathbf P^0\in B\})\,\rho(dp)\\
   &=\int_B\mu^p(A)\,\rho(dp).
 \end{aligned}
\end{equation}
Taking \(A=\mathcal X_\Gamma\) shows that
\(\Theta_\mu=\rho\). Since \(p\mapsto\mu^p\) is a Borel kernel,
\eqref{eq:disintegration-fibre-identity} proves the conditional-law assertion.

\medskip\noindent
\emph{Step 3: Identification of the central tail.}
The finite bridges identify the conditional law given the tail. Fix a
weighted prefix \(\gamma\) ending at \(P_\kappa\), where
\(\kappa\vdash r\), and let \(N\ge r+1\). Apply
\eqref{eq:prefix-uniformity-followup} at size \(N\): each prefix
ending at \(t\in\T_N\) has conditional probability
\(\mathbf1_{\{X_N=t\}}/u(t)\) given \(\mathscr G_N^0\).
There are \(n_{\mathrm{pf}}(\kappa;\lambda)\) such prefixes to
\(P_\lambda\) extending \(\gamma\). Summing over these prefixes and
\(\lambda\vdash N-1\) gives, \(\mu\)-almost surely,
\begin{equation}\label{eq:finite-bridge-tail-conditional}
 \E_\mu[\mathbf1_{C_\gamma}\mid\mathscr G_N^0]
 =\frac{n_{\mathrm{pf}}(\kappa;\lambda(X_N))}{u(X_N)}
 =\frac{K^{\mathrm{pf}}_\kappa(\lambda(X_N))}{u(P_\kappa)}.
\end{equation}
The right side is set to zero when \(X_N\) is not path-forest.
These variables lie in \([0,1]\). The endpoint Martin convergence and
\eqref{eq:path-forest-kernel-paintbox-limit} show that they converge
almost surely and in \(L^1(\mu)\) to
\(\chi_{\mathbf P_\mu}(\widetilde m_\kappa)
 =\mu^{\mathbf P_\mu}(C_\gamma)\).
The fields \(\mathscr G_N^0\) decrease to \(\Tcen^0\), so the
reverse-martingale convergence theorem yields
\(\E_\mu[\mathbf1_{C_\gamma}\mid\Tcen^0]
 =\mu^{\mathbf P_\mu}(C_\gamma)\).
Cylinders with a non-path-forest endpoint have zero probability under
\(\mu\) and every \(\mu^p\). The prefix cylinders, together with the empty set, form a generating
\(\pi\)-system. A monotone-class argument therefore extends the identity
to every Borel path event \(A\):
\begin{equation}\label{eq:raw-tail-conditional-kernel}
 \E_\mu[\mathbf1_A\mid\Tcen^0]=\mu^{\mathbf P_\mu}(A)
 \quad\mu\text{-a.s.}
\end{equation}
Since \(\mathbf P^0\) is raw-tail measurable,
\(\sigma(\mathbf P_\mu)\subseteq\Tcen^0\). For \(A\in\Tcen^0\),
\eqref{eq:raw-tail-conditional-kernel} gives
\(\mathbf1_A=\mu^{\mathbf P_\mu}(A)\) almost surely. Thus \(A\)
differs by a null set from the \(\sigma(\mathbf P_\mu)\)-measurable
event \(\{\mu^{\mathbf P_\mu}(A)>1/2\}\). Completing the two fields
proves their equality in \eqref{eq:conditional-tail-law}.

\medskip\noindent
\emph{Step 4: Uniqueness of the decomposition into extreme measures.}
Let \(\Upsilon\) be a representing probability measure for \(\mu\)
concentrated on the extreme points of \(\mathscr C(\Gamma)\).
The exposing functional \eqref{eq:exposing-functional} satisfies
\[
 0=\mathscr D(\mu)=\int_{\mathscr C(\Gamma)}
                         \mathscr D(\nu)\,\Upsilon(d\nu).
\]
Nonnegativity forces \(\Upsilon(\mathscr C_{\mathrm{pf}})=1\).
The extreme points of a face are precisely the extreme points of the
containing set that lie in the face. Therefore \(\Upsilon\) is
concentrated on \(\{\mu^p:p\in\King\}\). Pulling it back under the
homeomorphism \(p\mapsto\mu^p\) and using injectivity of
\(\mathsf B\) in \eqref{eq:Kingman-barycentre-map} gives the unique
measure \(\Theta_\mu\). By the face property and the Bauer representation,
\(\mu\) is extreme in \(\mathscr C(\Gamma)\) precisely when
\(\Theta_\mu\) is a point mass, equivalently when \(\mathbf P_\mu\)
is almost surely constant.
\end{proof}

\section{Finite-mass characters and Hopf convolution}
\label{sec:finite-mass-convolution}

The grading extends the normalized classification to arbitrary total mass. To describe the resulting parameter space, let
\[
 \mathfrak K=\{(c,q):c\ge0,\ q_1\ge q_2\ge\cdots\ge0,
                         \ \sum_iq_i\le c\},
\]
with the product topology in the total mass \(c\) and the ranked atom
coordinates. For \((c,q)\in\mathfrak K\), let
\(\operatorname{ev}_{c,q}:\Lambda\to\R\) be the unital algebra
homomorphism determined by
\[
 \operatorname{ev}_{c,q}(\mathsf p_1)=c,\qquad
 \operatorname{ev}_{c,q}(\mathsf p_k)=\sum_iq_i^k\quad(k\ge2),
\]
and put \(\Psi_{c,q}=\operatorname{ev}_{c,q}\circ\mathsf Q\). Its
dust mass is \(d=c-\sum_iq_i\).

\begin{theorem}[Finite mass and convolution]
\label{thm:positive-character-convolution}
The map \((c,q)\mapsto\Psi_{c,q}\) is a homeomorphism from
\(\mathfrak K\) onto the unital basis-positive characters, with pointwise
convergence on the tree basis. The parameter \((0,0)\) gives the counit;
for \(c>0\),
\begin{equation}\label{eq:finite-mass-scaling}
 \Psi_{c,q}(F_t)=c^{|t|-1}\psi_p(F_t),\qquad p_i=q_i/c.
\end{equation}
The zero map is the only nonunital basis-positive character.

The unital basis-positive characters form a commutative topological monoid
under Hopf convolution. Let \(q\sqcup r\) be the decreasing enumeration
of the multiset union of their positive entries, with zero padding when
this union is finite. Then
\begin{equation}\label{eq:positive-character-convolution}
 \Psi_{c,q}\star\Psi_{c',r}=\Psi_{c+c',\,q\sqcup r}.
\end{equation}
Equivalently, convolution takes the multiset union of the positive atom
masses, while the dust masses add.
\end{theorem}

\begin{proof}
Every nonzero character is unital. Let \(\psi\) be unital and
basis-positive, and put \(c=\psi(F_\tau)\ge0\). For \(m\ge1\),
\eqref{eq:iterated-Pieri} gives
\begin{equation}\label{eq:finite-mass-basis-bound}
 c^m=\sum_{t\in\T_{m+1}}u(t)\psi(F_t),
 \qquad
 0\le\psi(F_t)\le\frac{c^m}{u(t)}
       \quad(t\in\T_{m+1}).
\end{equation}
If \(c=0\), \eqref{eq:finite-mass-basis-bound} forces all
positive-degree values to vanish, so \(\psi=\varepsilon\).

For \(c>0\), define the linear functional \(\widehat\psi\) by
\(\widehat\psi(F_t)=c^{-(|t|-1)}\psi(F_t)\). The grading of the
product gives
\[
 \widehat\psi(F_sF_t)
 =c^{-(|s|-1)-(|t|-1)}\psi(F_sF_t)
 =\widehat\psi(F_s)\widehat\psi(F_t).
\]
Thus \(\widehat\psi\) is a normalized basis-positive character.
Theorem~\ref{thm:positive-GL-characters} yields a unique \(p\in\King\)
with \(\widehat\psi=\psi_p\). Set \(q_i=cp_i\). Conversely, by
homogeneity of \(\widetilde m_\lambda\), the specialization
\(\operatorname{ev}_{c,q}\) satisfies \eqref{eq:finite-mass-scaling}.
For \(c=0\), the defining inequalities force \(q=0\), and the
specialization is the counit. This proves existence, positivity, and
uniqueness for all parameters.

For continuity, on a bounded mass set \(c\le C\), the ordering of the
atom masses gives
\begin{equation}\label{eq:finite-mass-moment-tail}
 \sum_{i>L}q_i^k
 \le q_{L+1}^{k-1}\sum_{i>L}q_i
 \le\frac{C^k}{(L+1)^{k-1}},\qquad k\ge2.
\end{equation}
Every convergent parameter sequence has bounded total mass, so this uniform
tail estimate makes all moments, and hence all character values, continuous
on \(\mathfrak K\). For fixed \(C\), the set
\(\{(c,q)\in\mathfrak K:c\le C\}\) is compact: the ordering and
finite partial-sum inequalities make it closed in
\([0,C]^{\{0\}\cup\N}\). The parametrization is a homeomorphism on
each such set. A convergent sequence of characters has bounded values at
\(F_\tau\), so its parameters and the parameter of its limit lie in
one such set and therefore converge. Both spaces are metrizable, proving
continuity of the inverse.

The convolution of two unital characters is a unital character because
\(\Delta\) is an algebra morphism and the scalar target is commutative.
The coproduct \eqref{eq:GL-coproduct} has nonnegative coefficients and
is cocommutative, so convolution preserves basis positivity and is
commutative. Since \(\mathsf Q\) is a Hopf morphism and
\[
 \Delta\mathsf p_k=\mathsf p_k\otimes1+1\otimes\mathsf p_k,
 \qquad k\ge1,
\]
convolution adds all power-sum values. The parameter on the right of
\eqref{eq:positive-character-convolution} has the same values, so the
characters agree. Its dust mass is
\((c-\sum_iq_i)+(c'-\sum_ir_i)\). Finally, for every fixed \(x\),
\(\Delta x\) is a finite sum of simple tensors, which proves continuity
of convolution in the pointwise topology.
\end{proof}

\begin{corollary}[The common kernel of basis-positive characters]
\label{cor:positive-character-kernel}
The common kernel of all basis-positive characters is \(\Ipf\).
\end{corollary}

\begin{proof}
The vanishing theorem gives one inclusion. For the other, let
\(x\notin\Ipf\) and \(f=\mathsf Q(x)\ne0\). Choose \(N\ge1\) at
least the largest number of parts occurring in its monomial-basis expansion.
The monomials of distinct partition types have disjoint supports, so
\(f(x_1,\ldots,x_N)\) is a nonzero symmetric polynomial. It cannot
vanish on the open positive orthant, so some positive vector
\((q_1,\ldots,q_N)\) has \(f(q)\ne0\). Rank it and put
\(c=\sum_iq_i\). Ordinary evaluation at this vector is
\(\operatorname{ev}_{c,q}\), and \(\Psi_{c,q}(x)=f(q)\ne0\).
\end{proof}

On \(\mathfrak K\), write
\((c,q)\star(c',r)=(c+c',q\sqcup r)\) for the operation corresponding
to convolution. Let \(\mathfrak N_1\) be the Radon point measures on
\((0,\infty)\), with integer multiplicities and finite first moment
\(\int x\,\eta(dx)<\infty\). For \((c,q)\in\mathfrak K\), set
\(\eta_q=\sum_{i:q_i>0}\delta_{q_i}\). For \(n\ge2\), an \(n\)-th
convolution root of \(\Psi\) is a unital basis-positive character
\(\Xi\) such that
\(\Xi^{\star n}=\Psi\). Infinite divisibility means that such a root
exists for every integer \(n\ge2\). A convolution semigroup is a family
\((\Psi_t)_{t\ge0}\) in this monoid satisfying
\(\Psi_0=\varepsilon\) and \(\Psi_{s+t}=\Psi_s\star\Psi_t\).

\begin{theorem}[Atomic factorization, convolution roots, and semigroups]
\label{thm:positive-character-atomic-rigidity}
Let \((c,q)\in\mathfrak K\) and \(d=c-\sum_iq_i\).
\begin{enumerate}[label=(\roman*)]
\item The map \((c,q)\mapsto(d,\eta_q)\) is an isomorphism of
commutative monoids from \(\mathfrak K\) onto
\(\R_{\ge0}\times\mathfrak N_1\), with componentwise addition.
Writing \((a)=(a,0,0,\ldots)\), one has the pointwise factorization
\begin{equation}\label{eq:positive-character-atomic-factorization}
 \Psi_{c,q}=\lim_{L\to\infty}
 \left(\Psi_{d,0}\star\Psi_{q_1,(q_1)}\star\cdots
                         \star\Psi_{q_L,(q_L)}\right).
\end{equation}
\item For an integer \(n\ge2\), an \(n\)-th convolution root exists
if and only if every atom
multiplicity of \(\eta_q\) is divisible by \(n\). The root is then
unique and equals \(\Psi_{c/n,r}\), where \(\eta_r=\eta_q/n\).
\item The infinitely divisible characters are precisely the pure-dust
characters \(\Psi_{c,0}\). Every convolution semigroup is
\begin{equation}\label{eq:all-convolution-semigroups}
 \Psi_t=\Psi_{at,0},\qquad t\ge0,
\end{equation}
for a unique \(a\ge0\), and is consequently pointwise continuous.
\end{enumerate}
\end{theorem}

\begin{proof}
Since \(\sum_iq_i<\infty\), only finitely many atoms lie above any
positive threshold. Thus \(\eta_q\) is Radon and its first moment is
\(\sum_iq_i\). Conversely, for \(\eta\in\mathfrak N_1\) and
\(\varepsilon>0\),
\[
 \eta([\varepsilon,\infty))
 \le \varepsilon^{-1}\int_{(0,\infty)}x\,\eta(dx)<\infty.
\]
Hence its atoms can be listed in nonincreasing order with multiplicity,
with zero padding if there are finitely many. This gives a ranked sequence
\(q\) with \(\sum_iq_i=\int x\,\eta(dx)\); the total mass is
recovered as \(c=d+\int x\,\eta(dx)\). The convolution formula becomes
\((d,\eta)+(d',\eta')=(d+d',\eta+\eta')\), proving the monoid
isomorphism. In \eqref{eq:positive-character-atomic-factorization}, the
finite convolution has parameter
\((d+\sum_{i\le L}q_i,(q_1,\ldots,q_L,0,\ldots))\). It converges to
\((c,q)\), so the factorization follows from
Theorem~\ref{thm:positive-character-convolution}.

An \(n\)-th root in these coordinates solves
\[
 d=nd_n,\qquad \eta_q=n\eta_n.
\]
The dust equation gives \(d_n=d/n\). An integer-valued solution of the
second equation exists if and only if the stated divisibility condition holds,
and is then uniquely \(\eta_n=\eta_q/n\). Its first moment is
\(n^{-1}\sum_iq_i\), so its total mass is \(c/n\). This proves (ii).
Thus an \(n\)-th root keeps each positive atom mass fixed and divides
its multiplicity by \(n\).

Each positive atom has finite multiplicity. No positive integer is divisible
by every \(n\ge2\), so infinite divisibility forces \(\eta_q=0\).
Conversely \(\Psi_{c,0}=(\Psi_{c/n,0})^{\star n}\) for every \(n\).

Finally, for a convolution semigroup and every \(t>0\),
\(\Psi_t=(\Psi_{t/n})^{\star n}\) for all \(n\ge2\). Thus
\(\Psi_t=\Psi_{c(t),0}\), where \(c(t)=\Psi_t(F_\tau)\ge0\).
The semigroup law gives \(c(s+t)=c(s)+c(t)\). Nonnegativity makes
\(c\) nondecreasing, since \(c(t)-c(s)=c(t-s)\ge0\) for
\(t\ge s\). Let~\(a=c(1)\). Additivity gives \(c(r)=ar\) for
nonnegative rational \(r\); approximating an arbitrary \(t\) from
below and above by rational numbers and using monotonicity yields
\(c(t)=at\). This proves \eqref{eq:all-convolution-semigroups}; continuity follows
from the parameter homeomorphism.
\end{proof}

The monoid is cancellative: dust masses and the finite multiplicity at
each positive atom mass cancel separately. The inverse and idempotent
equations in these nonnegative coordinates also show that the counit is
the only invertible element and the only idempotent.

\bibliographystyle{alpha}
\bibliography{references}
\end{document}